\documentclass{amsart}    
\usepackage{graphicx}
\usepackage{hyperref}
\usepackage{amssymb}
\usepackage{setspace}

\title{The Non-triviality of the Grope Filtrations of The Knot and Link Concordance Groups}

\author{Peter D Horn}
\address{Department of Mathematics\\Rice University\\6100 S. Main Street\\Houston, TX 77005 }
\email{phorn@rice.edu}
\urladdr{http://math.rice.edu/~phorn}

\newtheorem{thm}{Theorem}[section]    
\newtheorem{lem}[thm]{Lemma}          

\newtheorem{prop}[thm]{Proposition}
\newtheorem*{mainresult}{Theorem~\ref{mainresult}}
\newtheorem*{secondaryresult}{Theorem~\ref{secondaryresult}}

\theoremstyle{definition}
\newtheorem{defn}[thm]{Definition}    

\newtheorem{rem}{Remark}             
\newcommand{\p}{\pi_1}
\newcommand{\del}{\partial}
\renewcommand{\a}{\alpha}
\renewcommand{\b}{\beta}
\newcommand{\lk}{{\mathrm{lk}}}
\newcommand{\Gn}[1]{G^{(#1)}}

\renewcommand{\H}[2]{H_{#1}\left({#2}\right)}
\newcommand{\pn}[2]{\p\left({#1}\right)^{({#2})}}
\newcommand{\F}{{\mathcal{F}}}
\newcommand{\C}{{\mathcal{C}}}
\newcommand{\R}{{\mathbb{R}}}
\newcommand{\Z}{{\mathbb{Z}}}
\newcommand{\Q}{{\mathbb{Q}}}
\newcommand{\G}{{\mathcal{G}}}
\newcommand{\N}{{\mathbb{N}}}
\newcommand{\BG}{{\mathcal{BG}}}
\newcommand{\BF}{{\mathcal{BF}}}
\newcommand{\B}{{\mathcal{B}}}

\begin{document}

\maketitle

\begin{abstract}    
	We consider the Grope filtration of the classical knot concordance group that was introduced in a paper of Cochran, Orr and Teichner.  Our main result is that successive quotients at each stage in this filtration have infinite rank.  We also establish the analogous result for the Grope filtration of the concordance group of string links consisting of more than one component.
\end{abstract}

\section{Introduction}

	We work in the smooth category.  A knot is a smooth  embedding of the circle in the three-sphere.  Two knots are isotopic if there is a smooth isotopy of the three-sphere  taking one knot the the other.  Isotopy is a three-dimensional notion of knot equivalence.  In this paper, we study a four-dimensional notion of knot equivalence called \emph{concordance}.  Two knots $K_{0}$ and $K_{1}$ are concordant if there exists a smooth embedding of an annulus $A=S^{1}\times[0,1]$ into $S^{3}\times[0,1]$ such that $A\cap\left(S^{3}\times\{i\}\right)=K_{i}$ for $i=0,1$.  One can check that isotopic knots are concordant, and one can find counterexamples to the converse.  Given two knots, one can combine them via the connect sum operation, denoted $\#$.  See~\cite[p 39]{dR76} for more details.  The set of knots modulo concordance, equipped with the connect sum operation, forms an abelian group.  We wish to study this \emph{concordance group} $\C$.

	In~\cite{COT03}, Cochran, Orr, and Teichner introduced two filtrations of the knot concordance group by subgroups.  The $(n)$-solvable filtration $$\cdots\subset \F_{n.5}\subset\F_{n}\subset\cdots\subset\F_{1}\subset\F_{0.5}\subset\F_{0}\subset\C$$ is defined loosely by $K\in\F_{n}$ if $0$-surgery on $K$ bounds a $4$-manifold whose first and second  homology groups satisfy certain conditions related to the $n$-th derived subgroup of that $4$-manifold.  The precise definition will be provided in Section~\ref{definitions}.  It is true that a slice knot lies in $\F_{n}$ for all $n\in\frac{1}{2}\N$.  In other words, this filtration measures how far \emph{algebraically} a knot is from being slice (take, for example, the highest $n$ for which $K\in\F_{n}$ but $K\not\in\F_{n.5}$).
	
	The other filtration defined in~\cite{COT03} is the Grope filtration $$\cdots\subset \G_{n.5}\subset\G_{n}\subset\cdots\subset\G_{3}\subset\G_{2.5}\subset\G_{2}\subset\G_{1.5}\subset\C$$ which is defined by $K\in\G_{n}$ if $K$ bounds a Grope of height $n$ in $D^{4}$.  A Grope is a $2$-complex that will be defined in Section~\ref{definitions}.  As in the $(n)$-solvable filtration, a slice knot lies in $\G_{n}$ for all $n$, and so the highest $n$ for which $K\in\G_{n}$ but $K\not\in\G_{n.5}$ is some measure on how far \emph{geometrically} $K$ is from being slice.
	
	The non-triviality of the $(n)$-solvable filtration has been a topic of study since its introduction.  That $\F_{n}/\F_{n.5}$ has infinite rank was proved by Jiang in~\cite{bJ81} for $n=1$, by Cochran, Orr and Teichner in~\cite[Theorem 4.1]{COT04} for $n=2$, and by Cochran, Harvey and Leidy in~\cite[Theorem 8.1]{CHL07b} for all $n$.  Our main result is an analogous statement for the Grope filtration.
	
	\begin{mainresult}
		For every $n\geq 2$, $\G_{n}/\G_{n.5}$ has infinite rank.
	\end{mainresult}
	
	Our main result improves upon previous results regarding the non-triviality of the Grope filtration.  For example, Cochran and Teichner proved in~\cite[Theorem 1.4]{CT07} the existence of knots of infinite order in $\G_{n}/\G_{n.5}$.  Our method of constructing knots bounding Gropes is similar to that of Cochran and Teichner.  Since $\G_{n+2}\subset\F_{n}$~\cite[Theorem 8.11]{COT03}, we are able to adapt Cochran, Harvey and Leidy's technique~\cite{CHL07b} to prove that no non-trivial linear combination of these knots lies in $\F_{n.5}$, hence not in $\G_{n+2.5}$.
	
	\begin{rem}
		We construct knots that bound smooth embeddings of gropes in $D^{4}$.  We state our main result so that is easy to read, but as stated, it is not the strongest possible assertion.  The result of Cochran, Harvey and Leidy implies that no non-trivial linear combination of these knots is even topologically $(n.5)$-solvable, hence does not bound even a topologically flat embedding of a height $n.5$ grope in $D^{4}$.
	\end{rem}
	
	 For the Grope filtration of the group of $m$-component string links with $m>1$, Harvey showed in~\cite[Theorem 6.13]{sH08} that for all $n$, $\G^{m}_{n}/\G^{m}_{n+2}$ contains an infinitely generated subgroup.  Let $\BG^{m}_{n}$ denote the subgroup of $\G^{m}_{n}$ whose elements are boundary links.  Combining~\cite[Theorem 6.13]{sH08} and~\cite[Theorem 4.4]{CH08}, it is known that for each $m>1$ and $n\geq 2$, the abelianization of $\BG^{m}_{n}/\BG^{m}_{n+1.5}$ has infinite rank.  We improve the result by narrowing the index gap.
	 
	 \begin{secondaryresult}
	 	For each $m>1$ and $n\geq 2$, the abelianization of $\BG^{m}_{n}/\BG^{m}_{n.5}$ has infinite rank; hence $\G_{n}^{m}/\G^{m}_{n.5}$ contains an infinitely generated subgroup.
	\end{secondaryresult}

\section{Definitions}\label{definitions}

If $G$ is a group, the \textbf{derived series of $G$} is defined recursively by $\Gn{0}=G$ and $\Gn{i+1}=\left[\Gn{i},\Gn{i}\right]$.

Let $M$ be closed, connected, orientable 3-manifold.  Recall from~\cite{COT03} the definition of an $(n)$-solution for $M$.

\begin{defn}
	A smooth, spin 4-manifold $W$ with $\del W=M$ is an \textbf{$(n)$-solution} for $M$ if the inclusion-induced map $i_\ast:\H{1}{M}\to\H{1}{W}$ is an isomorphism and if there are embedded surfaces $L_i$ and $D_i$ (with product neighborhoods) for $i=1,\ldots,m$ that satisfy the following conditions:
	\begin{enumerate}
		\item the homology classes $\left\{\left[L_1\right],\left[D_1\right],\ldots,\left[L_m\right],\left[D_m\right]\right\}$ form an ordered basis for $\H{2}{W}$,
		\item the intersection form $\big(\H{2}{W},\cdot\big)$ with respect to this ordered basis is a direct sum of hyperbolics,
		\item $L_i\cap D_j$ is empty if $i\neq j$,
		\item for each $i$, $L_i$ and $D_i$ intersect transversely at a point, and
		\item each $L_i$ and $D_i$ are \textbf{$(n)$-surfaces}, i.e. $\p(L_i)\subset\pn{W}{n}$ and $\p(D_i)\subset\pn{W}{n}$.
	\end{enumerate}
	If, in addition, $\p(L_{i})\subset\pn{W}{n+1}$ for each $i$, we say $W$ is an \textbf{$(n.5)$-solution} for $M$.
	
	There is a slight generalization of $n$-solvability, called \textbf{rational $n$-solvablility}.  For the purposes of this paper, it suffices to know that an $n$-solution is also a rational $n$-solution.  If the above conditions hold except for the isomorphism on first homology, $W$ is called an \textbf{$(n)$-bordism} for $M$.
	
	If a closed, orientable 3-manifold bounds an $(n)$-solution, we say $M$ is \textbf{$(n)$-solvable}.  A knot $K$ in $S^3$ is said to be an \textbf{$(n)$-solvable knot} if $0$-surgery on $K$ is $(n)$-solvable.
\end{defn}

As in~\cite{COT03}, the set of all $(n)$-solvable knots is denoted $\F_n$, where it was shown that these $\F_n$ filter the topological knot concordance group $\C$, i.e. the $\F_n$ form a nested series of subgroups of $\C$: $$\{0\}\subset\cdots\subset\F_{n.5}\subset\F_n\subset\cdots\subset\F_{1.5}\subset\F_1\subset\F_{0.5}\subset\F_0\subset\C$$

The following definition is modified from the definition in~\cite{FT95}:

\begin{defn}
	A \textbf{grope} is a special pair (2-complex, base circle).  A grope has a \textbf{height} $n\in\frac{1}{2}\mathbb{N}$.  A grope of height $1$ is precisely a compact, oriented surface $\Sigma$ with a single boundary component (the base circle).  For $n\in\N$, a grope of height $n+1$ is defined recursively as follows: let $\{\a_i,\b_{i}: i=1,\ldots,2(g-1)\}$ be a symplectic basis of curves for $\Sigma$, the first stage of the grope.  Then a grope of height $n+1$ is formed by attaching gropes of height $n$ to each $\a_i$ and $\b_{i}$ along the base circles.
	
	A grope of height $1.5$ is formed by attaching gropes of height $1$ (i.e. surfaces) to a Lagrangian of a symplectic basis of curves for $\Sigma$.  That is, a grope of height $1.5$ is a surface with surfaces glued to `half' of the basis curves (i.e. the $\a_{i}$).  In general, a grope of height $n+1.5$ is obtained by attaching gropes of height $n$ to the $\a_{i}$ and gropes of height $n+1$ to the $\beta_{i}$.

	Given a 4-manifold $W$ with boundary $M$ and a framed circle $\gamma\subset M$, we say that $\gamma$ bounds a \textbf{Grope} (note the capital `G') in $W$ if $\gamma$ extends to a smooth embedding of a grope with its untwisted framing.  That is, parallel push-offs of Gropes can be taken in $W$.  Knots in $S^3$ are always equipped with the zero framing.
\end{defn}

	Cochran, Orr, and Teichner~\cite{COT03} also defined the Grope filtration of $\C$.  The set of all knots that bound Gropes of height $n$ in $D^4$ is denoted $\G_n$, which is a subgroup of $\C$.  Observe the following connection between gropes and the derived series: if a curve $\ell$ bounds a (map of a) grope of height $n$ in a space $X$, then $[\ell]\in\pn{X}{n}$.  It was shown in~\cite[Theorem 8.11]{COT03} that $\G_{n+2}\subset\F_n$ (for all $n\in\frac{1}{2}\N$).

	We will also use a few definitions from~\cite{CT07}.

\begin{defn}
	An \textbf{annular grope of height $n$} is a grope of height $n$ that has an extra boundary component on its first stage.  We say that the two boundary components of an annular grope cobound an annular grope.  Two knots $K_{0}$ and $K_{1}$ are \textbf{height $n$ Grope concordant} if they cobound a height $n$ annular Grope $G$ in $S^{3}\times[0,1]$ such that $G\cap\left(S^{3}\times\{i\}\right)=K_{i}$ for $i=0,1$.  For example, if $K$ is height $n$ Grope concordant to a slice knot, then $K\in\G_{n}$.  The captial `G' in `annular Grope' indicates the untwisted framing.
\end{defn}

	Let $G$ be a height $n$ Grope in $S^{3}\times [0,1]$ bounded by a knot $K\subset S^{3}\times\{0\}$.  The union of sets of basis curves for each $n$-th stage surface of $G$ is called a \textbf{set of tips} for $G$.  If the tips lie on annuli in $S^{3}\times[0,1]$ that are disjoint from each other and disjoint from the Grope (except at the tips) and the other boundary components of the annuli lie in $S^{3}\times\{1\}$, we say `the tips are isotopic' to the the link determined by the other boundary components of these annuli (a link in $S^{3}\times\{1\}$).
	
	A \textbf{capped Grope} is the union of a Grope and a disjoint union of discs where the boundaries of the discs form a full set of tips of the Grope.  The interiors of these discs must not intersect the Grope except perhaps on one of the boundary components of the Grope's first stage.  These discs are called the `caps' of the Grope.  For example, if $K$ bounds a Grope, the caps are allowed to hit $K$.
	
	The construction of our examples relies on a technique known as \emph{infection}.  Let $R$ be a fixed knot or link and $K$ be a fixed knot in $S^{3}$.  Suppose $\alpha$ is a simple closed curve in $S^{3}-R$ such that $\alpha$ is itself the unknot.  Some number $m$ of strands of $R$ pierce the disc bounded by $\alpha$.  Let $R(\alpha,K)$ denote the knot obtained by replacing the $m$ trivial strands of $R$ by $m$ strands `tied into the knot K.'  See Figure~\ref{infection}.  More precisely, one obtains $R(\alpha,K)$ by removing a regular neighborhood of $\alpha$ and gluing in $S^{3}-K$ in such a way that identifies the meridian of $\alpha$ with the longitude of $K$ and the longitude of $\alpha$ with the meridian of $K$.  We say $R(\alpha,K)$ is the result of infecting $R$ by $K$ along $\alpha$.
	
	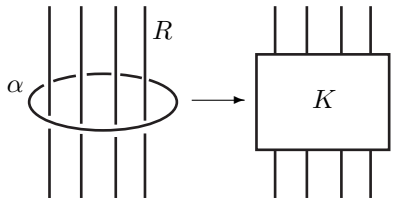
\begin{figure}[ht!]
		\begin{center}
			\begin{picture}(140, 75) (0,0)
				\put(-8, 40){$\alpha$}
				\put(47, 60){$R$}
				\put(62, 37){\vector(1,0){20}}
				\includegraphics{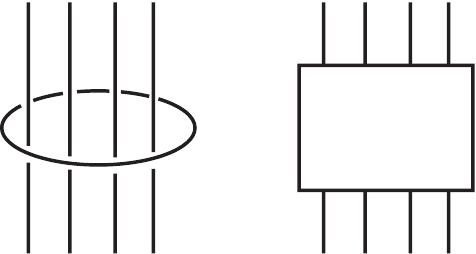}
				\put(-30, 34){$K$}
			\end{picture}
			\caption{Infecting $R$ by $K$ along $\alpha$}
			\label{infection}
		\end{center}
	\end{figure}
	
	Let $(H,J)$ denote the zero-framed, ordered Hopf link in the complement of a knot or link $R$ in Figure~\ref{hopf}.  Surgery on $(H,J)$ is homeomorphic to $S^{3}$~\cite[\S 4, Proposition 2]{rK78}, so the image of $R$ after performing this Hopf surgery is a new knot or link in $S^{3}$.  By sliding the strands of $R$ (that pass through $J$) over $H$ and cancelling the Hopf pair, one sees the effect of Hopf surgery in Figure~\ref{hopf}.  We will use the terminology `performing the handle cancellation $J-H$' to mean `performing $(H,J)$ Hopf surgery.'  This convention matches~\cite[p 350]{CT07}.
	
	\begin{figure}[ht!]
		\begin{center}
			\begin{picture}(165, 75)(0,0)
				\put(-8, 35){$J$}
					\put(0, 22){$0$}
				\put(70, 40){$H$}
					\put(65, 20){$0$}
				\put(75, 32){\vector(1,0){20}}
				\put(-2, 60){$R$}
				\includegraphics{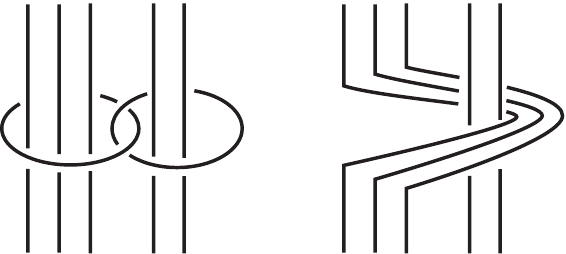}
			\end{picture}
			\caption{The effect of a Hopf surgery}
			\label{hopf}
		\end{center}
	\end{figure}
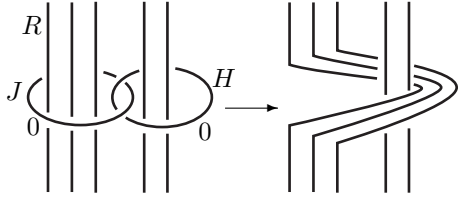

\section{Knots Bounding Gropes of Height $n+2$}\label{highgropes}

	Fix an $m\geq 1$.  For each $n\geq 1$, we will construct a knot $K^{n}_{m}$ by infecting a seed knot $R$ along certain curves in $\pn{S^{3}-R}{1}$ by a knot $K_{m}^{n-1}$.  Roughly speaking, we construct a height $n+2$ Grope bounded by $K^{n}_{m}$ in two stages.  First we prove each $K_{m}^{n-1}$ bounds a height $n+1$ Grope whose tips are isotopic to meridians of $K_{m}^{n-1}$.  When we infect $R$ by $K_{m}^{n-1}$, we will see that the tips of this Grope are isotopic to the infection curves, which we will then prove bound capped Gropes of height $1$.  Gluing these height $1$ capped Gropes onto the tips of the height $n+1$ Gropes for $K_{m}^{n-1}$ (and gluing a few annuli and slice discs) will result in a height $n+2$ Grope bounded by $K^{n}_{m}$.

\subsection{The infecting knots $P_{m}$}\label{infectingknots}

	Let $P_{m}$ be the knot defined in Figure~\ref{Pm}.  A priori, we have described $P_{m}$ as a knot in some closed, oriented $3$-manifold that is the image of $U$ after performing $0$-surgery on the $(10+4m)$-component link in Figure~\ref{Pm}.  Using Kirby's calculus of framed links (see, for example,~\cite[Chapter 5]{GS99}), one can check that this $3$-manifold is $S^{3}$.

\begin{figure}[ht!]
	\begin{center}
		\begin{picture}(375, 110)(0, 0)
			\put(54, 4){$\cdots$}
			\put(54, 20){$\cdots$}
			\put(54, 70){$\cdots$}
			\put(12, 0){$\underbrace{\hspace{4.8cm}}$}
			\put(76, -13){$m$}
			\put(262, 90){$P_{m}$}
			\put(358, 54){$P_{m}$}
			\put(275, 54){\Large$\sim$}
			\put(11, 68){\small$0$}
			\put(31, 68){\small$0$}
			\put(69, 68){\small$0$}
			\put(89, 68){\small$0$}
			\put(111, 68){\small$0$}
			\put(131, 68){\small$0$}
			\put(161, 68){\small$0$}
			\put(200, 70){\small$0$}
			\put(11, 38){\small$0$}
			\put(32, 38){\small$0$}
			\put(69, 38){\small$0$}
			\put(90, 38){\small$0$}
			\put(111, 38){\small$0$}
			\put(132, 38){\small$0$}
			\put(160, 50){\small$0$}
			\put(203, 50){\small$0$}
			\put(168, 28){\small$0$}
			\put(15, 6){\small$0$}
			\put(181, 10){\small$0$}
			\put(205, 33){\small$0$}
			\put(237, 33){\small$0$}
			\put(240, 68){\small$0$}
			\includegraphics{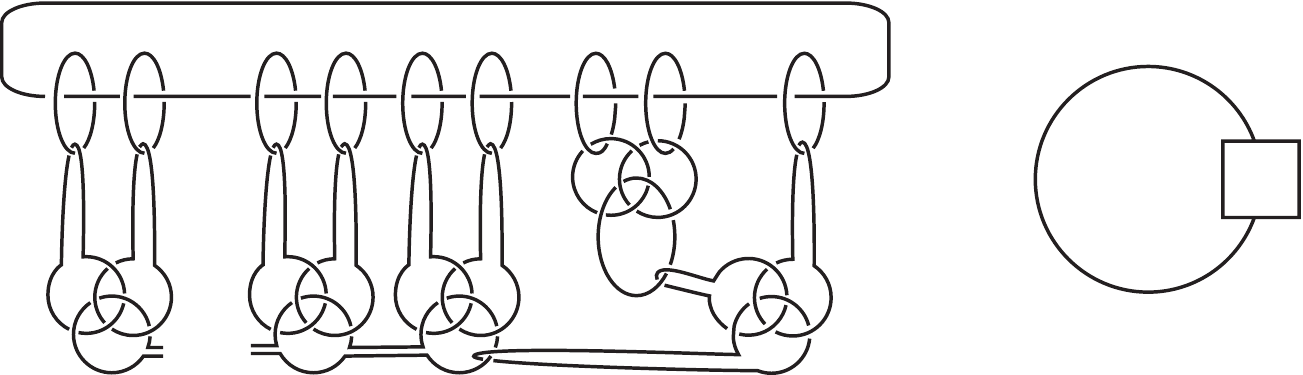}
		\end{picture}
		\caption{A surgery description of $P_{m}$}
		\label{Pm}
	\end{center}
\end{figure}
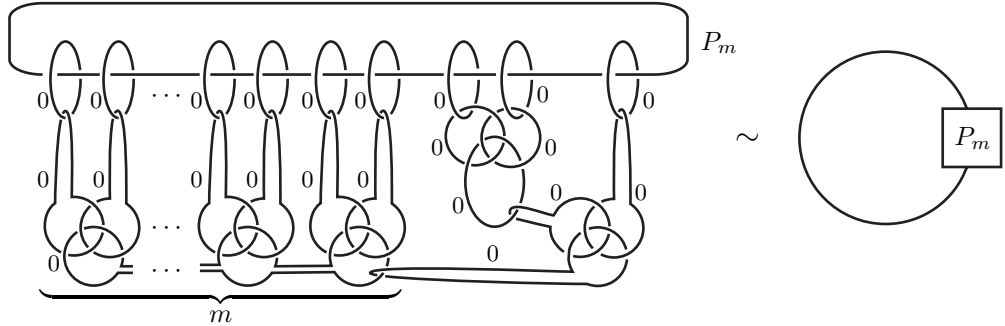

\begin{prop}\label{grope}
	Each $P_{m}$ cobounds with the unknot a height $2$ annular Grope in in $S^{3}\times[0,1]$ whose tips are isotopic to a set of disjoint copies of meridians of $P_{m}$.
\end{prop}
\begin{proof}
	We claim that Figure~\ref{gropepicture} depicts a height $2$ annular grope in $S^{3}$ between $P_{m}$ and $U$.  Let $L$ denote the $(10+4m)$-component link $L=(H_{0},\ldots, H_{9+4m})$ in Figure~\ref{firstgropepicture}; the link $L$ is isotopic to the link defining $P_{m}$ in Figure~\ref{Pm}.  In Figure~\ref{firstgropepicture}, we have shaded in some punctured surfaces bounded by the components $H_{0}$, $H_{4}$, and $H_{9}$.  Once we perform $0$-surgery on $L$, these punctured surfaces will become closed surfaces, and these surfaces will be used to construct a height $2$ Grope.
	\begin{figure}[ht!]
		\begin{center}
			\begin{picture}(260, 210)(0,0)
				\put(72, 93){$\cdots$}
				     \put(252, 100){\tiny $H_{0}$}
				     \put(248, 170){\tiny $H_{1}$}
				     \put(172, 70){\tiny $H_{2}$}
				     \put(202, 70){\tiny $H_{3}$}
				     \put(212, 110){\tiny $H_{4}$}
				     \put(182, 150){\tiny $H_{5}$}
				     \put(216, 150){\tiny $H_{6}$}
				     \put(180, 170){\tiny $H_{7}$}
				     \put(214, 170){\tiny $H_{8}$}
				     \put(118, 85){\tiny $H_{9}$}
				     \put(85, 150){\tiny $H_{10}$}
				     \put(147, 150){\tiny $H_{11}$}
				     \put(82, 170){\tiny $H_{12}$}
				     \put(145, 170){\tiny $H_{13}$}
				     \put(-8, 150){\tiny $H_{6+4m}$}
				     \put(-10, 170){\tiny $H_{8+4m}$}
				     \put(38, 200){\tiny $H_{9+4m}$}
				     \put(42, 198){\vector(2,-3){7}}
				     \put(68, 198){\tiny $H_{7+4m}$}
				     \put(76, 194){\vector(-1,-4){9}}
				\put(10, 70){$\underbrace{\hspace{4.7cm}}$}
				\put(72, 55){$m$}
				\includegraphics{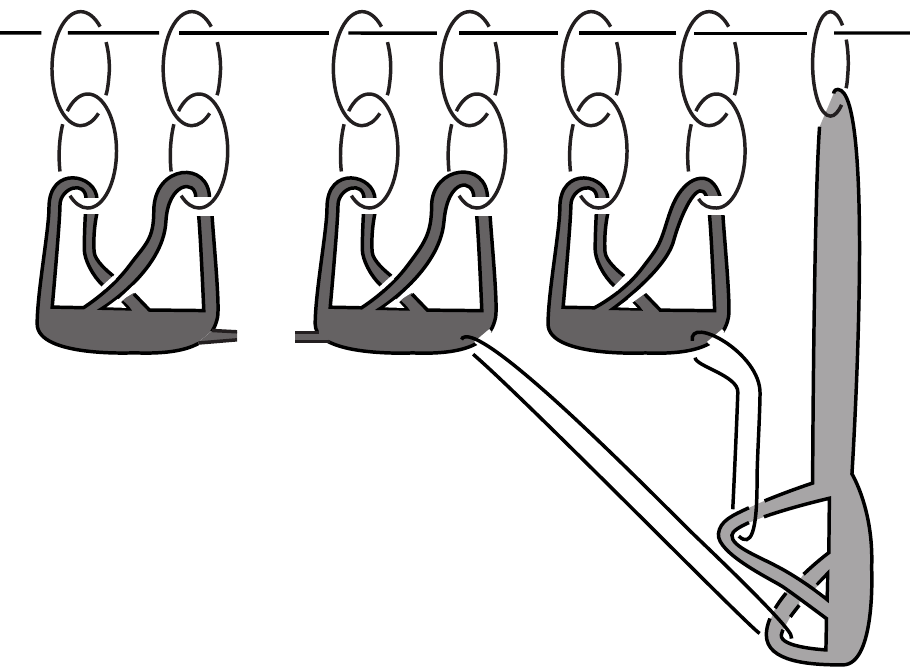}
			\end{picture}
			\caption{The link $L$}
			\label{firstgropepicture}
		\end{center}
	\end{figure}
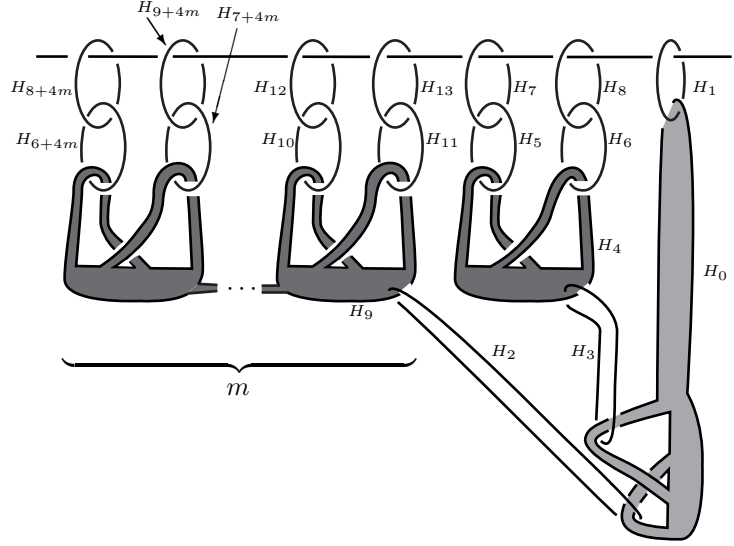
	
	We perform $2m+2$ handle cancellations,
	$$H_{7}-H_{5},\hspace{3mm} H_{8}-H_{6},\hspace{3mm}  H_{12}-H_{10},\hspace{3mm}  H_{13}-H_{11},\hspace{3mm}  \ldots, H_{8+4m}-H_{6+4m},\hspace{3mm}  H_{9+4m}-H_{7+4m},$$
	to arrive at Figure~\ref{gropepicture} (ignore the $U$ and shaded parts for now).  What remains is a $6$-component link $N=(H_{0},H_{1}, H_{2}, H_{3},H_{4},H_{9})$.  Let $N'$ denote the link $N-\{H_{0},H_{1}\}$; $N'$ consists of two separated Hopf links, thus $0$-surgery on $N'$ (denoted $S^{3}_{N'}$) is homeomorphic to $S^{3}$.  Let $U$ be the knot in Figure~\ref{gropepicture} that  deviates only slightly from $P_{m}$.  The resulting knot $U'$ after $0$-surgery on $N'$ is the unknot in $S^{3}\cong S^{3}_{N'}$, since $U$ does not hit the discs bounded by $H_{1}$, $H_{2}$ and $H_{3}$.
	
	There is a punctured annulus in $S^{3}$ bounded by $U$ and $P_{m}$; the lightest shade of grey in Figure~\ref{gropepicture} is \emph{part} of this annulus.  The other part of this annulus has been collapsed to simplify the diagram.  To the puncture, glue one boundary component of a tube.  Push this annulus over $H_{1}$, and glue the other boundary component to the circle where $H_{1}$ punctures the torus on which $H_{0}$ lies (second darkest shade of grey in Figure~\ref{gropepicture}).  So far we have a height $1$ (and genus $1$) annular grope between $P_{m}$ and the unknot in $S^{3}$.
	
	One may take the obvious symplectic basis elements for the aforementioned torus, tube them over $H_{2}$ and $H_{3}$ and attach them to punctured the surfaces for $H_{4}$ and $H_{9}$ (the darkest shade of grey in Figure~\ref{gropepicture}).  Recall that after zero surgery on the link, $H_{4}$ and $H_{9}$ lie on closed surfaces.  Now we have a height $2$ annular grope in $S^{3}$ between $P_{m}$ and the unknot.  Figure~\ref{gropepicture} shows that the tips of this annular grope are isotopic to meridians of $P_{m}$.
	\begin{figure}[ht!]
		\begin{center}
			\begin{picture}(260, 200)(0,0)
					\put(266, 180){$P_{m}$}
					\put(266, 200){$U$}
					\put(252, 100){\tiny $H_{0}$}
					\put(250, 170){\tiny $H_{1}$}
					\put(172, 70){\tiny $H_{2}$}
					\put(202, 70){\tiny $H_{3}$}
					\put(212, 110){\tiny $H_{4}$}
					\put(118, 85){\tiny $H_{9}$}
					\put(72, 93){$\cdots$}
					\put(10, 70){$\underbrace{\hspace{4.7cm}}$}
					\put(72, 55){$m$}
					\includegraphics{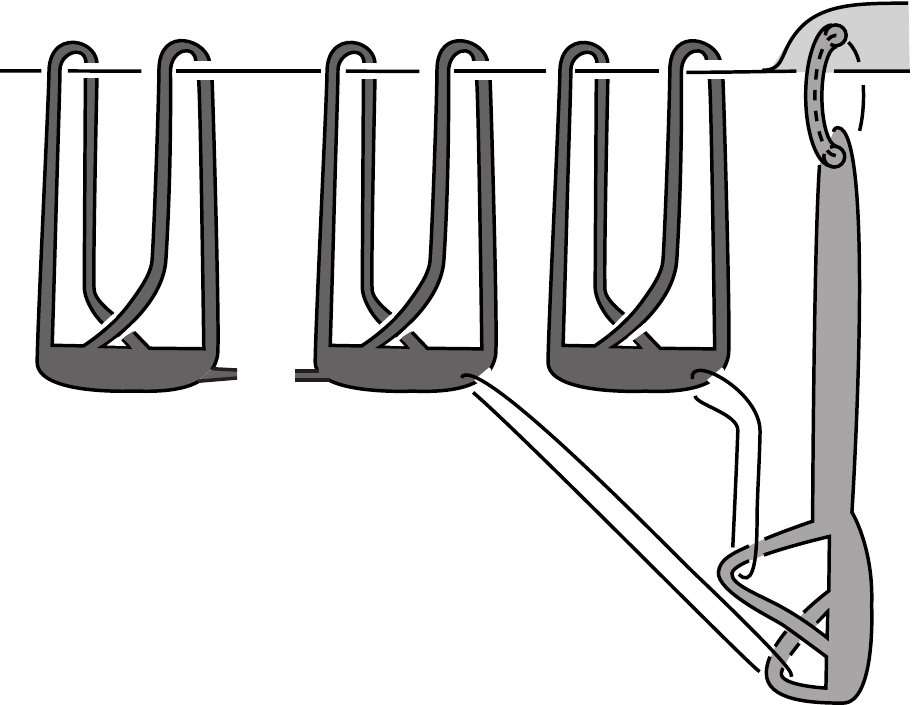}
			\end{picture}
			\caption{A height $2$ Grope concordance from the unknot $U$ to $P_{m}$}
			\label{gropepicture}
		\end{center}
	\end{figure}
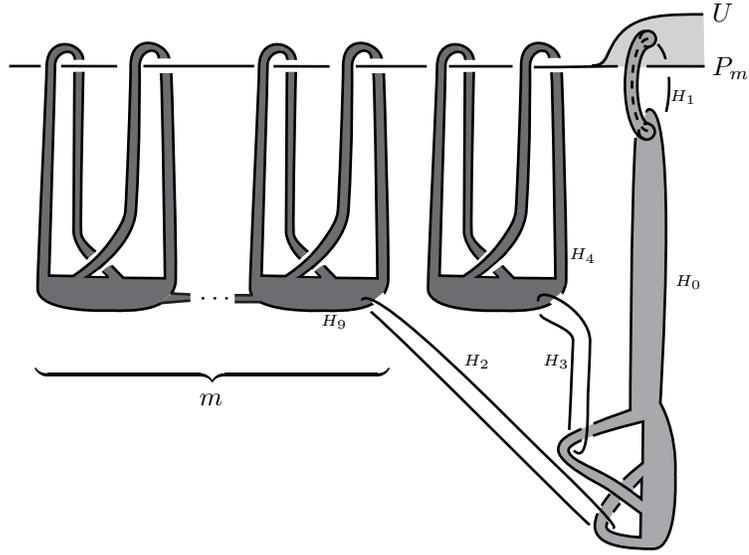
	
	One can push the $U$ boundary of component of this annular grope in the positive $I$ direction and the $P_{m}$ boundary component in the negative $I$ direction to arrive at the desired result.
\end{proof}

\subsection{The seed knot $R$ and the infection curves}\label{seeds}

Our examples are inspired by Cochran, Harvey and Leidy's family $J_{n}$.  They construct their family inductively by infecting the $9_{46}$ knot along the curves $\a$ and $\b$ in Figure~\ref{infectingcurves}.  Note that the obvious tori in $S^{3}-9_{46}$ bounded by $\a$ and $\b$ intersect.  For our construction of the Gropes for $K^{n}_{m}$, it is necessary that our infection curves bound disjoint surfaces in the complement of $9_{46}$.  We homotope $\a$ and $\b$ to arrive at different infection curves $\a'$ and $\b'$.

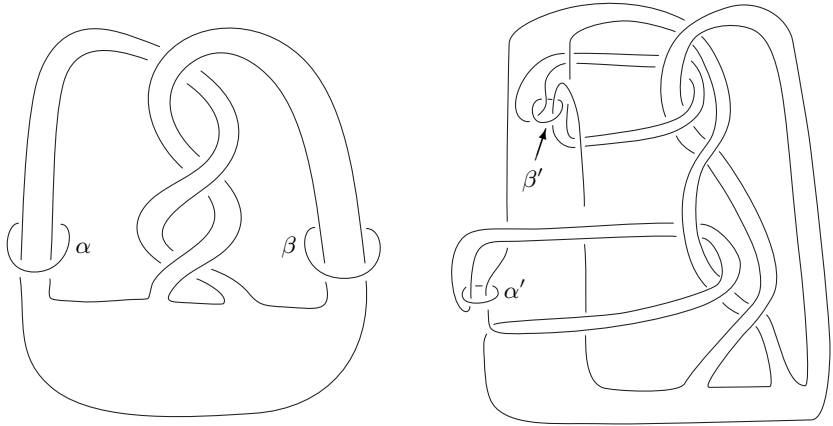
\begin{figure}[h!]
	\begin{center}
		\begin{picture}(330, 160)(0,0)
			\put(26, 65){\small$\alpha$}
			\put(104, 65){\small$\beta$}
			\put(188, 48){\small$\alpha'$}
			\put(195, 90){\small$\beta'$}
			\put(200, 100){\vector(1,3){4}}
			\includegraphics[bb=0 0 777 398, scale=.4]{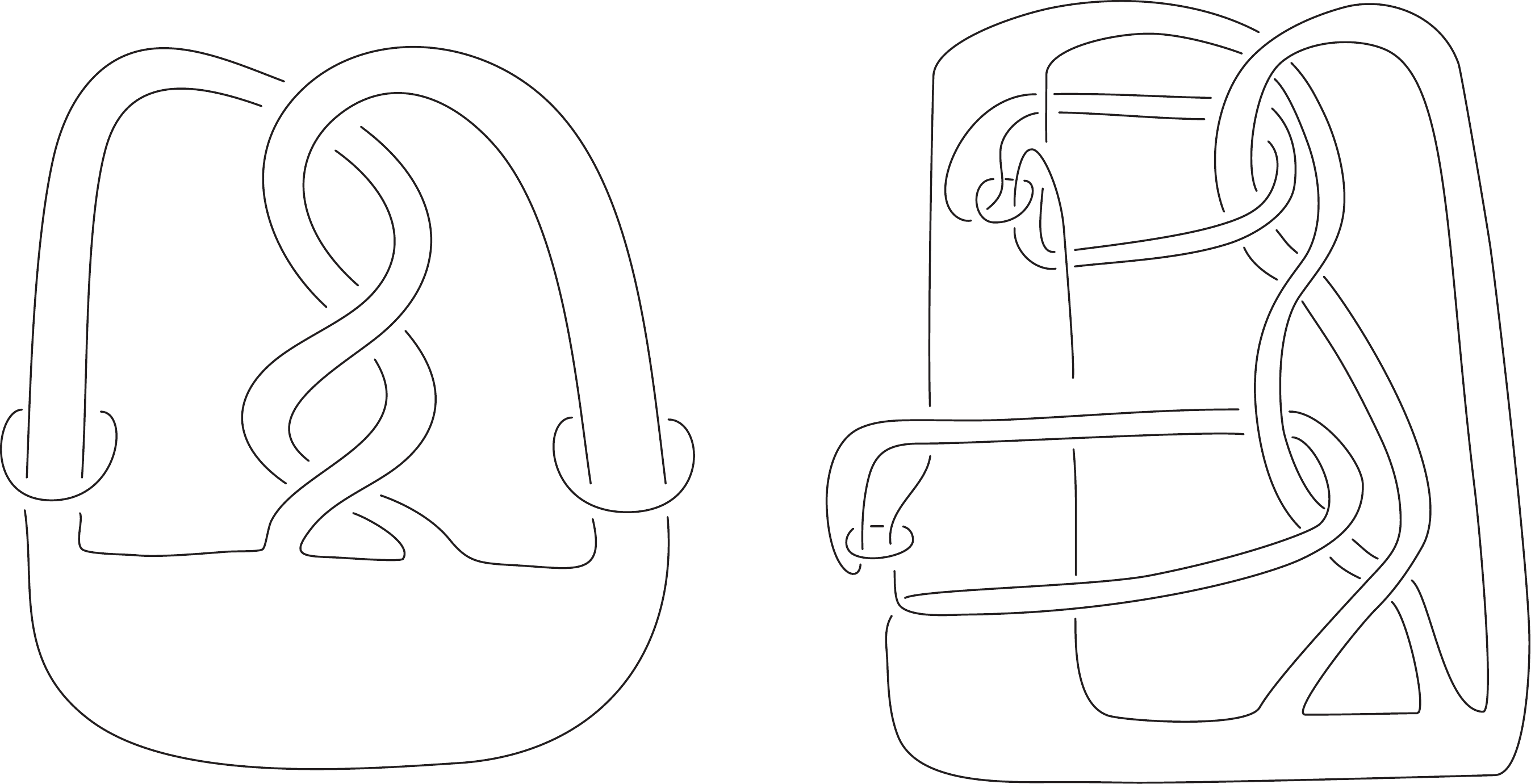}
		\end{picture}
		\caption{The infection curves $\alpha$ and $\beta$, and homotopic infection curves $\alpha'$ and $\beta'$}
		\label{infectingcurves}
	\end{center}
\end{figure}

Let $R=9_{46}$.  By Figure~\ref{curvesboundcappedtori}, the curves $\a'$ and $\b'$ bound tori that are disjoint and that miss $R$.  The torus bounded by $\a'$ has a symplectic basis $x_{1},x_{2}$ (the curve $x_{2}$ has been isotoped off of the torus) of curves that inherit the zero framing from the torus.  The torus bounded by $\b'$ has a similar symplectic basis $y_{1},y_{2}$.  It is clear from Figure~\ref{curvesboundcappedtori} that the curves $x_{1},x_{2},y_{1},y_{2}$ bound disjoint discs that hit the knot $R$ transversely.  Thus, $\a'$ and $\b'$ bound disjoint capped tori in the complement of $R$.

\begin{figure}[ht!]
	\begin{center}
		\begin{picture}(280, 295)(0, 0)
			\put(34, 100){\small $\a'$}
			\put(46, 223){\small $\b'$}
			\put(8, 80){\small $x_{1}$}
			\put(132, 120){\small $x_{2}$}
			\put(56, 199){\small $y_{1}$}
			\put(105, 230){\small $y_{2}$}
			\includegraphics[scale=.7]{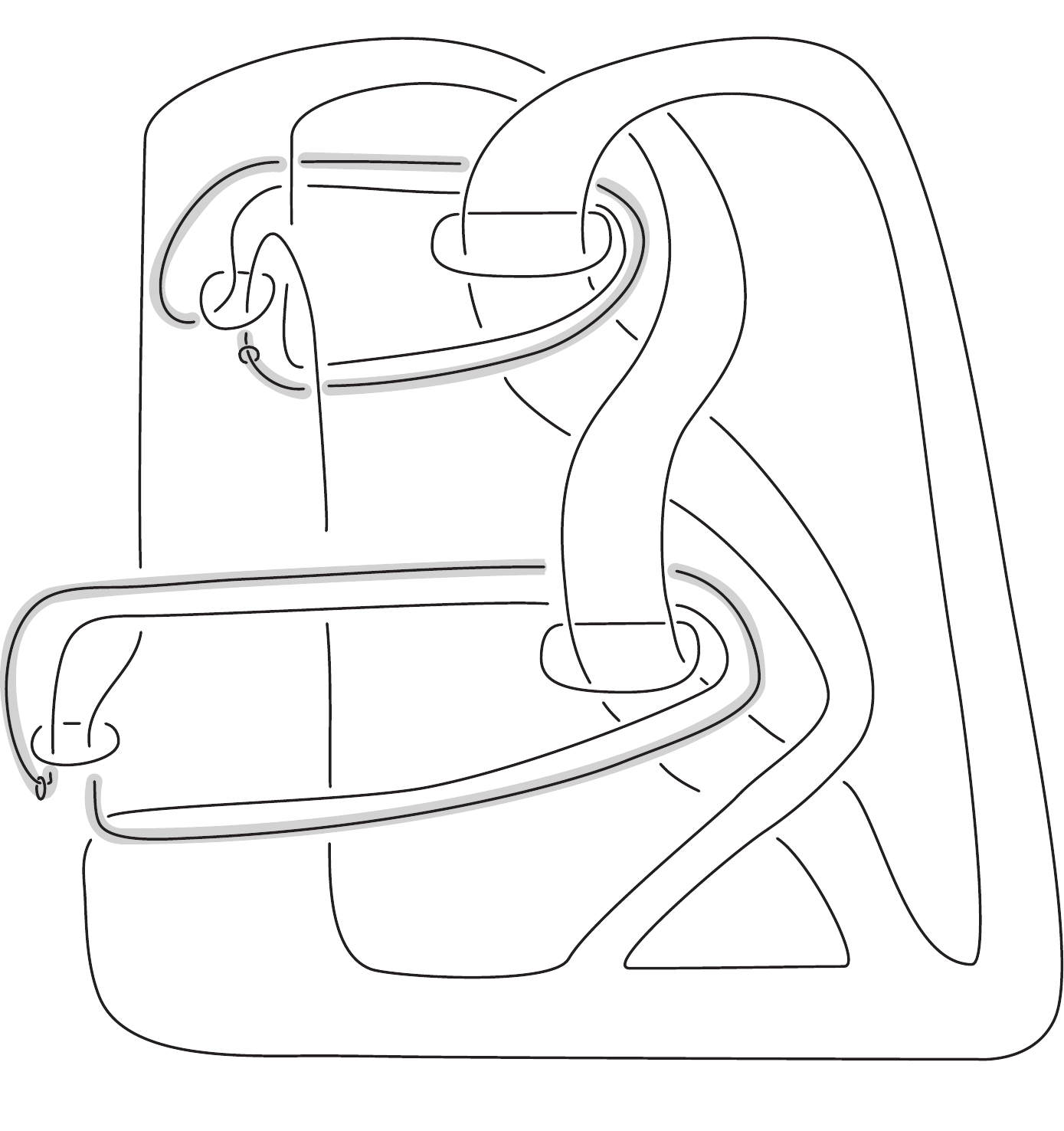}
		\end{picture}
		\caption{Infection curves bound disjoint capped tori}
		\label{curvesboundcappedtori}
	\end{center}
\end{figure}

\subsection{Knots bounding Gropes of height $n+2$}\label{ourexamples}

	We now describe our family of knots $K^{n}_{m}$ for $n,m\geq 1$ with the property that $K^{n}_{m}\in{\mathcal{G}}_{n+2}$.  Recall $R=9_{46}$, and $\alpha'$ and $\beta'$ be from Figure~\ref{infectingcurves}.  The curves $\alpha'$ and $\beta'$ bound punctured tori that are disjointly embedded in $S^{3}-R$.  Let $K^{1}_{m}$ denote $R$ infected along $\alpha'$ and $\beta'$ by our knot $P_{m}$.  Define $K^{n+1}_{m}$ by infecting $R$ along $\alpha'$ and $\beta'$ by $K^{n}_{m}$.   Observe that $g(K^{n}_{m})\leq 3$ since $\alpha'$ and $\beta'$ miss a genus three Seifert surface for $R$.

\begin{prop}\label{heightngropeconcordinfection}
	Let $R$ be a knot in $S^{3}$, $\a$ an unknotted curve in $S^{3}-R$.  If $K$ and $J$ are knots that are height $n$ Grope concordant, then the results of infection $R(\a,K)$ and $R(\a,J)$ are height $n$ Grope concordant.
\end{prop}
\newpage
\begin{proof}[Sketch of Proof.]
	The knots $R(\a,K)$ and $R(\a,J)$ differ by what happens inside a fixed ball $D^{3}$:
	\begin{center}
		\begin{picture}(122, 83)(0,0)
			\includegraphics{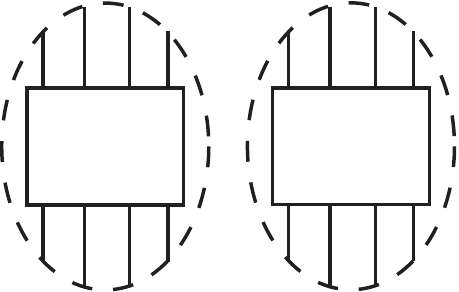}
			\put(-106, 38){$K$}
			\put(-34, 38){$J$}
		\end{picture}
	\end{center}
	Let $(D^{3},m\,K)$ denote the left-hand side of this diagram, i.e. a 3-ball with $m$ arcs tied into $K$.  Since $K$ and $J$ are height $n$ Grope concordant, so are $(D^{3},m\,K)$ and $(D^{3},m\,J)$.  We see this by slitting a height $n$ Grope concordance from $K$ to $J$, which turns it into a `height $n$ slit Grope concordance' from one knotted arc to another.  Taking parallel copies of this slit Grope concordance in $D^{3}\times[0,1]$ yields $G$, the `height $n$ slit Grope concordance' in $D^{3}\times[0,1]$ from $(D^{3},m\,K)\times\{0\}$ to $(D^{3},m\,J)\times\{1\}$.  We simply glue this $G$ to $\left(R-\left(R\cap D^{3}\right)\right) \times[0,1]$ to arrive at a height $n$ Grope concordance from $R(\a,K)$ to $R(\a,J)$.
\end{proof}

\begin{prop}\label{basecase}
	Let $R=9_{46}$ and $K^{1}_{m}$ be defined as above.  Then $K^{1}_{m}$ is height $3$ capped Grope concordant to $R$.
\end{prop}
\begin{proof}
	For simplicity, let $K=K^{1}_{m}$.  $K$ is $R$ infected along $\a'$ and $\b'$ by $P_{m}$.  We will first concentrate on the infection along the $\a'$ curve.  Let $K'$ denote $R$ infected along $\a'$ by $P_{m}$.  We will infect along $\b'$ later, arriving at $K$.
	Consider a diagram for the link $R,\a'$.  The knot $K'$ can be described by replacing the infection curve with the framed $(10+4m)$-component link that defines $P_{m}$.  Figure~\ref{infectionoperation} depicts this operation.
	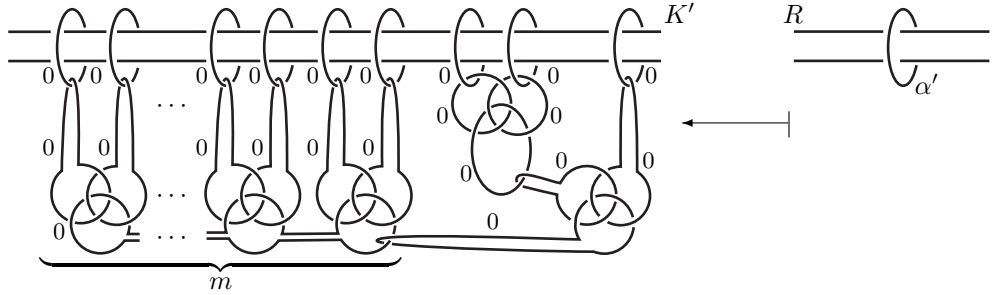
\begin{figure}[ht!]
		\begin{center}
			\begin{picture}(375, 110)(0, 0)
			\put(56, 4){$\cdots$}
			\put(56, 20){$\cdots$}
			\put(56, 54){$\cdots$}
			\put(12, 0){$\underbrace{\hspace{4.8cm}}$}
			\put(76, -13){$m$}
			\put(13, 65){\small$0$}
			\put(31, 65){\small$0$}
			\put(71, 65){\small$0$}
			\put(91, 65){\small$0$}
			\put(113, 65){\small$0$}
			\put(133, 65){\small$0$}
			\put(163, 65){\small$0$}
			\put(204, 65){\small$0$}
			\put(13, 38){\small$0$}
			\put(34, 38){\small$0$}
			\put(71, 38){\small$0$}
			\put(92, 38){\small$0$}
			\put(113, 38){\small$0$}
			\put(134, 38){\small$0$}
			\put(162, 50){\small$0$}
			\put(205, 50){\small$0$}
			\put(171, 28){\small$0$}
			\put(17, 6){\small$0$}
			\put(181, 10){\small$0$}
			\put(207, 33){\small$0$}
			\put(240, 33){\small$0$}
			\put(241, 65){\small$0$}
			\put(343, 61){$\a'$}
			\put(293, 88){$R$}
			\put(248, 88){$K'$}
			\put(295, 50){\vector(-1,0){40}}
			\put(294, 47){$|$}
			\includegraphics{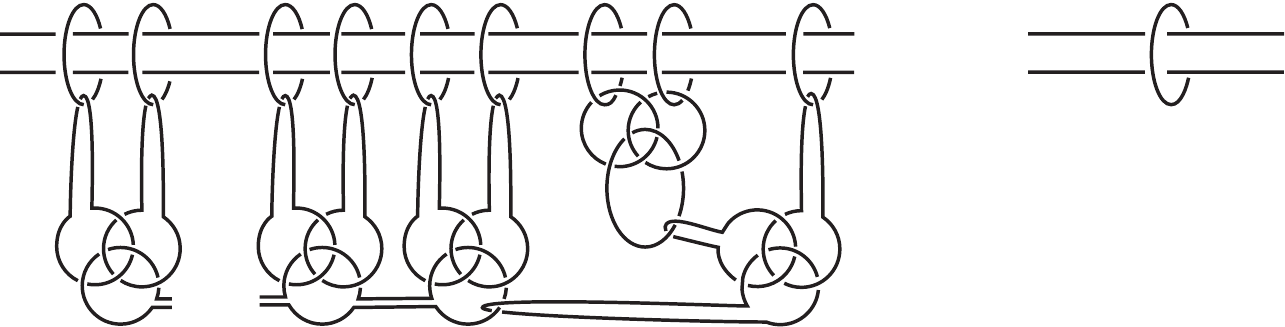}
		  \end{picture}
			\caption{The knot $K$ as infection on $R$}
			\label{infectionoperation}
		\end{center}
	\end{figure}
	
	One may form a new knot $L$ by pulling the $2$ strands of $K'$ out of the upper-right-most $0$-framed curve in Figure~\ref{infectionoperation}.  This is an `$2$-stranded analogue' of what is done in Figure~\ref{gropepicture}.  As in Proposition~\ref{grope}, there is a punctured annulus in $S^{3}$ between $K'$ and $L$ (it is punctured twice).  These punctures are $2$ parallel copies of a meridian of the $0$-framed curve through which we pulled the $2$ strands of $K'$.  Float these punctures downward in the $I$ direction so that each puncture is in a different $S^{3}\times\{\ast\}$ level.  To each puncture, one can glue the height $2$ Grope for $P_{m}$ described in Proposition~\ref{grope}.  As before, the tips of the Grope in the $S^{3}\times\{\ast\}$ level are isotopic to parallel copies of the infection curve $\a'\times\{\ast\}$.  Now, we may float each of these $\a'\times\{\ast\}$ disjointly down to the $S^{3}\times\{1\}$ level.  Since each of the pieces of the Grope (and the tips) miss $K'$ and $L$, we may glue $L\times[0,1]$ to this Grope.  The result is a height $2$ Grope concordance from $K'\times\{0\}$ to $L\times\{1\}$ whose tips are disjointly isotopic to parallel copies of $\a'\times\{1\}$.
	
	We claim that $L$ is in fact $R$.  The proof of this fact is similar to the proof that the $U$ from Figure~\ref{gropepicture} and Proposition~\ref{grope} is the unknot.  Namely, we have expressed $L$ as surgery on separated Hopf links, and for each of these Hopf pairs, $L$ does not pass through one of the components.  Thus, performing the surgery does not change $L$ locally, and away from the surgery link, $L$ is identically $R$.
	
	At this point we have proven $K'=R(\a',P_{m})$ is height $2$ Grope concordant to $R$, and the tips of this Grope concordance (which we name $G$) are isotopic to parallel copies of $\a'$.  We can view $K$ as $K'(\b',P_{m})$.  An argument similar to the one above yields a height $2$ Grope concordance (which we name $H$) between $K$ and $K'$ whose tips are isotopic to parallel copies of $\b'$.  We can glue $H\subset S^{3}\times[0,1]$ to $G\subset S^{3}\times[1,2]$ and float the tips of $H$ down through $S^{3}\times[1,2]$ so that they miss $G$.  In order to justify this, we recall from the proof of Proposition~\ref{heightngropeconcordinfection} that $G$ was constructed by gluing pieces of $R\times I$ (which miss $\b'$) to something that lived in $D^{3}\times I$, where $\a'\subset D^{3}$ and $\b'$ was contained in $S^{3}-D^{3}$.  Thus, $G\cap \left(\b'\times [1,2]\right)=\emptyset$.  We see that $J:=G\cup H\cup\left(\b'\times[1,2]\sqcup\cdots\sqcup\b'\times[1,2]\right)$ is a height $2$ Grope concordance between $K^{1}_{m}$ and $R$ with punctured caps, and these punctures are parallel copies of $\b'$ union parallel copies of $\a'$ such that the punctures lie in $\left(S^{3}-R\right)\times\{2\}$.  The result is a height $2$ Grope concordance between $K=K^{1}_{m}$ and $R$ whose tips are the union of parallel copies of $\a'$ and parallel copies of $\b'$.  A schematic of $J$ is depicted in Figure~\ref{schematicofgrope}, with $n=0$.  Pushing further down into $S^{3}\times[2,3]$, we can glue parallel copies of the capped tori to the copies $\a'$ and $\b'$.  The reader will recall that these capped tori miss $R$.  Since we attached a capped torus to each tip of $J$, we are left with a height $3$ capped Grope concordance between $K^{1}_{m}$ and $R$, as desired.
\end{proof}

\begin{thm}\label{knmingn+2}
	For each $n,m\geq 1$, $K^{n}_{m}$ is height $n+2$ capped Grope concordant to $R=9_{46}$.  In particular, $K^{n}_{m}$ bounds a Grope of height $n+2$ in $D^{4}$.
\end{thm}
\begin{proof}
	We induct on $n$.  The claim is true in the case $n=1$ by Proposition~\ref{basecase}.  Assume that the claim is true for $K^{n}_{m}$.
	
	By definition $K^{n+1}_{m}$ is $R$ infected along $\a'$ and $\b'$ by $K^{n}_{m}$.  As in the proof of Proposition~\ref{basecase}, we first concentrate on the infection along $\a'$.  Let $K'=R(\a',K^{n}_{m})$.  By the inductive hypothesis on $K^{n}_{m}$ and by Proposition~\ref{heightngropeconcordinfection}, $K'$ and $R=R(\a,\mbox{unknot})$ are height $n+2$ Grope concordant.  This Grope concordance is built using the height $n+2$ capped Grope concordance $L$ between $K^{n}_{m}$ and $R$.  Each cap of $L$ hits $R$ transversely, and removing neighborhoods of the intersection points from the cap yields a planar surface with boundary equal to a tip of $L$ and copies of the meridian of $R$.  The meridians of $R$ can float along $L$ to be meridians of $K^{n}_{m}$.  Under the infection of $R$ along $\a'$ by $K^{n}_{m}$, the meridian of $K^{n}_{m}$ is identified with (the longitude of a regular neighborhood of) $\a'$.  Let $G$ denote the union of $L$ and these planar surfaces reaching out from the tips.  By the `punctures of $G$,' we mean the punctures of these punctured caps, each of which is a copy of $\a'$.
	
	As for the infection along $\a'$, between $K^{n+1}_{m}=K'(\b',K^{n}_{m})$ and $K'$ there is a height $n+2$ Grope concordance $H$ with punctured caps whose punctures are copies of $\b'$.  Since infecting along $\b'$ is a local operation near $\b'$, $H$ misses $\a'\times[0,1]$.
	
	We now describe the height $n+3$ capped Grope concordance between $K^{n+1}_{m}$ and $R$.  Figure~\ref{schematicofgrope} shows a schematic of the ultimate Grope.  We start with the height $n+2$ Grope concordance $H$ between $K^{n+1}_{m}$ and $K'$.  Imagine that $H$ is embedded in $S^{3}\times[0,1]$ so that $H\cap S^{3}\times\{0\}=K^{n+1}_{m}$ and $H\cap S^{3}\times\{1\}=K'$.  The punctures of $H$ are parallel copies of $\b'$ in $S^{3}\times\{1\}$.  Now imagine that $G$, which is the Grope concordance between $K'$ and $R$, is imbedded in $S^{3}\times[1,2]$ so that $G\cap S^{3}\times\{1\}=K'$ and $G\cap S^{3}\times\{2\}=R$.  The punctures of $G$ are parallel copies of $\a'$ in $S^{3}\times\{2\}$.  We claim that we can `float down the punctures of $H$' into the $S^{3}\times\{2\}$ level \emph{while missing $G$}.  A justification of this was given in the proof of Proposition~\ref{basecase}.  We see that $J:=G\cup H\cup\left(\b'\times[1,2]\sqcup\cdots\sqcup\b'\times[1,2]\right)$ is a height $n+2$ Grope concordance between $K^{n+1}_{m}$ and $R$ with punctured caps, and these punctures are parallel copies of $\b'$ union parallel copies of $\a'$ such that the punctures lie in $\left(S^{3}-R\right)\times\{2\}$.

	Since $\a'$ and $\b'$ bound disjoint capped torus in $S^{3}-R$, we can glue parallel copies of the capped torus (in $S^{3}\times[2,3]$) to $J$ along the punctures of $J$.  To summarize, we have attached capped surfaces to each tip in the height $n+2$ Grope concordance $J$.  What we have is a height $n+3$ capped Grope concordance between $K^{n+1}_{m}$ and $R$.  We have proven the inductive step, and the claim is true for all $n$.
	
	\begin{figure}[ht!]
		\begin{center}
			\begin{picture}(420, 155)(0,0)
				\put(150, 130){$K^{n+1}_{m}$}
				\put(150, 76){$K'$}
				\put(150, 27){$R$}
				\put(38, -5){copies of $\a'$}
				\put(335, -5){copies of $\b'$}
				\put(275, 140){$\overbrace{\hspace{5cm}}$}
				\put(343, 150){$H$}
				\put(-5, 90){$\overbrace{\hspace{5cm}}$}
				\put(62, 100){$G$}
				\includegraphics{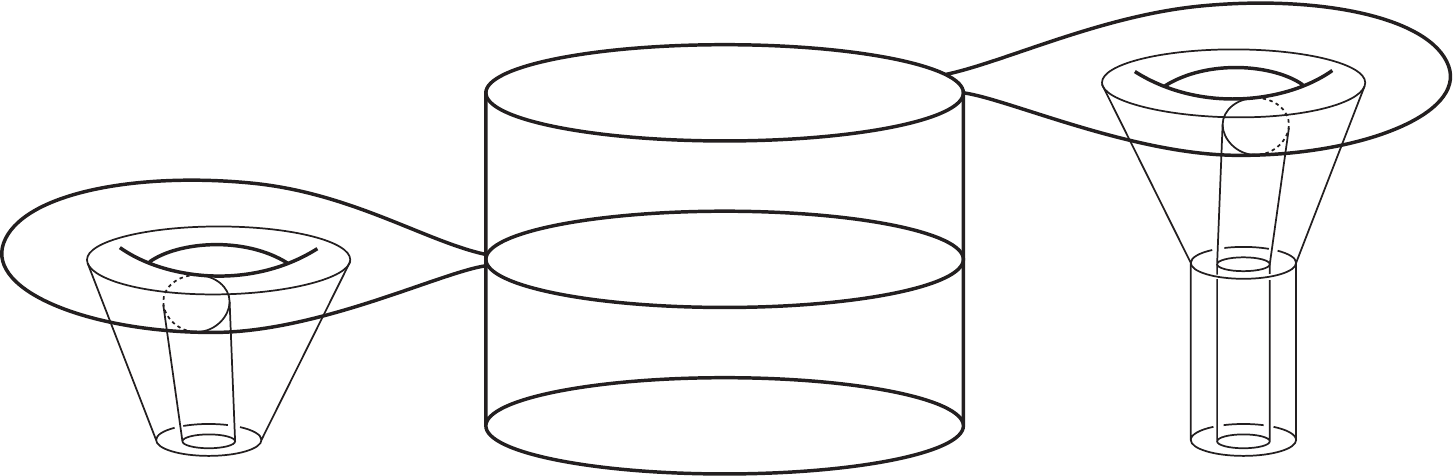}
			\end{picture}
			\caption{A schematic of the height $n+2$ Grope concordance, $J$, from $K^{n+1}_{m}$ to $R$}
			\label{schematicofgrope}
		\end{center}
	\end{figure}
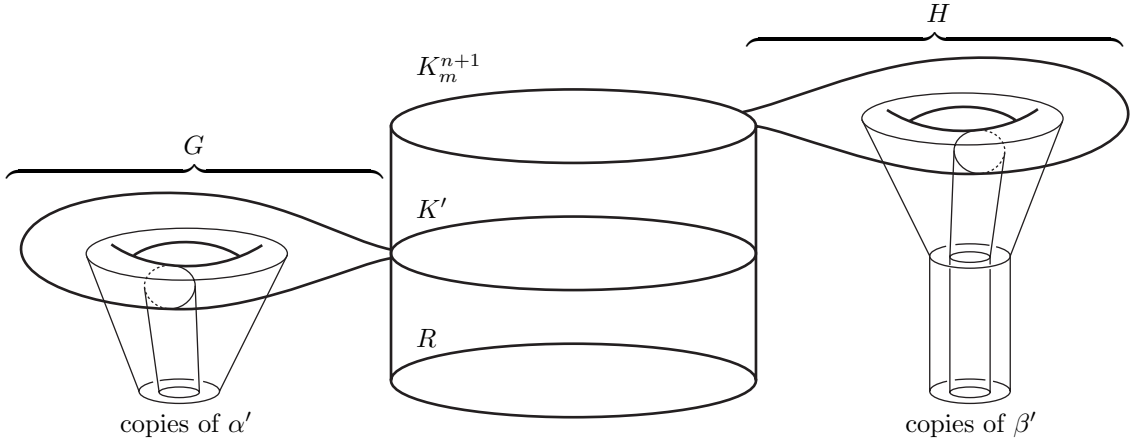
\end{proof}

\section{Algebraic Invariants}
	
	We outline a proof that $K^{n}_{m}$ (or a subsequence of similar knots) are linearly independent in $\G_{n+2}/\G_{n+2.5}$.  The following arguments are adapted from Cochran, Harvey and Leidy's proof that $\F_{n}/\F_{n.5}$ has infinite rank (cf.~\cite[Theorem 8.1]{CHL07b}).
	
\begin{defn}
	Let $K$ be a knot in $S^{3}$ and $G=\p\left(M_{K}\right)$.  The real number $\rho^{1}(K)$  is the first-order $L^{(2)}$-signature given by the Cheeger-Gromov invariant $\rho\left(M_{K},\phi:G\twoheadrightarrow G/G^{(2)}\right)$.
\end{defn}

\begin{thm}\label{mainresult}
	For $n\geq 0$, $\G_{n+2}/\G_{n+2.5}$ has infinite rank.
\end{thm}

We produce an infinite family of knots in $\G_{n+2}$ whose members are linearly independent in $\G_{n+2}/\G_{n+2.5}$.  This family is a subsequence of $\{K^{n}_{m}\}_{m=1}^{\infty}$, or possibly a subsequence of a related family (see the proof below for a discussion).  To avoid unnecessary reproduction of previous results, we will give a brief outline of the proof and refer the reader to~\cite[Theorem 8.1]{CHL07b} for the details.  We merely push through the proof given by Cochran, Harvey and Leidy in~\cite{CHL07b}, just to verify that their argument may be used to show that no non-trivial linear combination of the $K^{n}_{m}$ lie in $\F_{n.5}$, hence not in $\G_{n+2.5}$.

\begin{proof}[Sketch of proof]
	First, we must find a genus one ribbon knot $R$ with $\rho^{1}(R)\neq 0$.  If $\rho^{1}(9_{46})\neq 0$, we can set $R=9_{46}$.  Otherwise, we can tie one band of $9_{46}$ into the right-handed trefoil knot, as shown in Figure~\ref{addtrefoil}; the result, $R$, is a ribbon knot with $\rho^{1}(R)\neq 0$.
	
	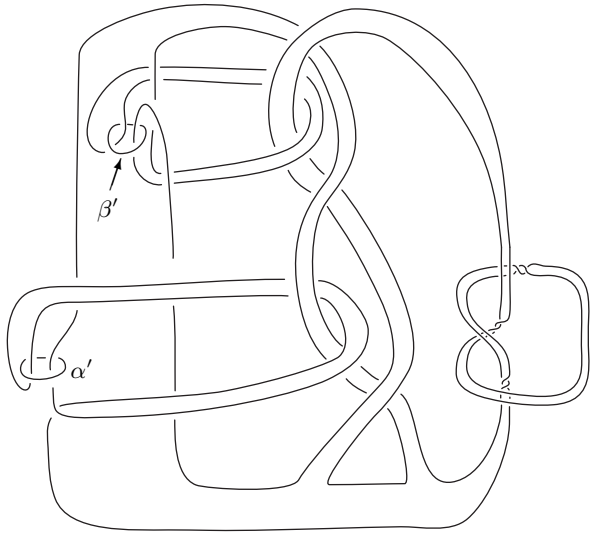
\begin{figure}[ht!]
		\begin{center}
			\begin{picture}(240, 200)(0, 0)
			\put(28, 58){\small$\alpha'$}
			\put(38, 119){\small$\beta'$}
			\put(43, 129){\vector(1,3){4}}
			\includegraphics[scale=.5]{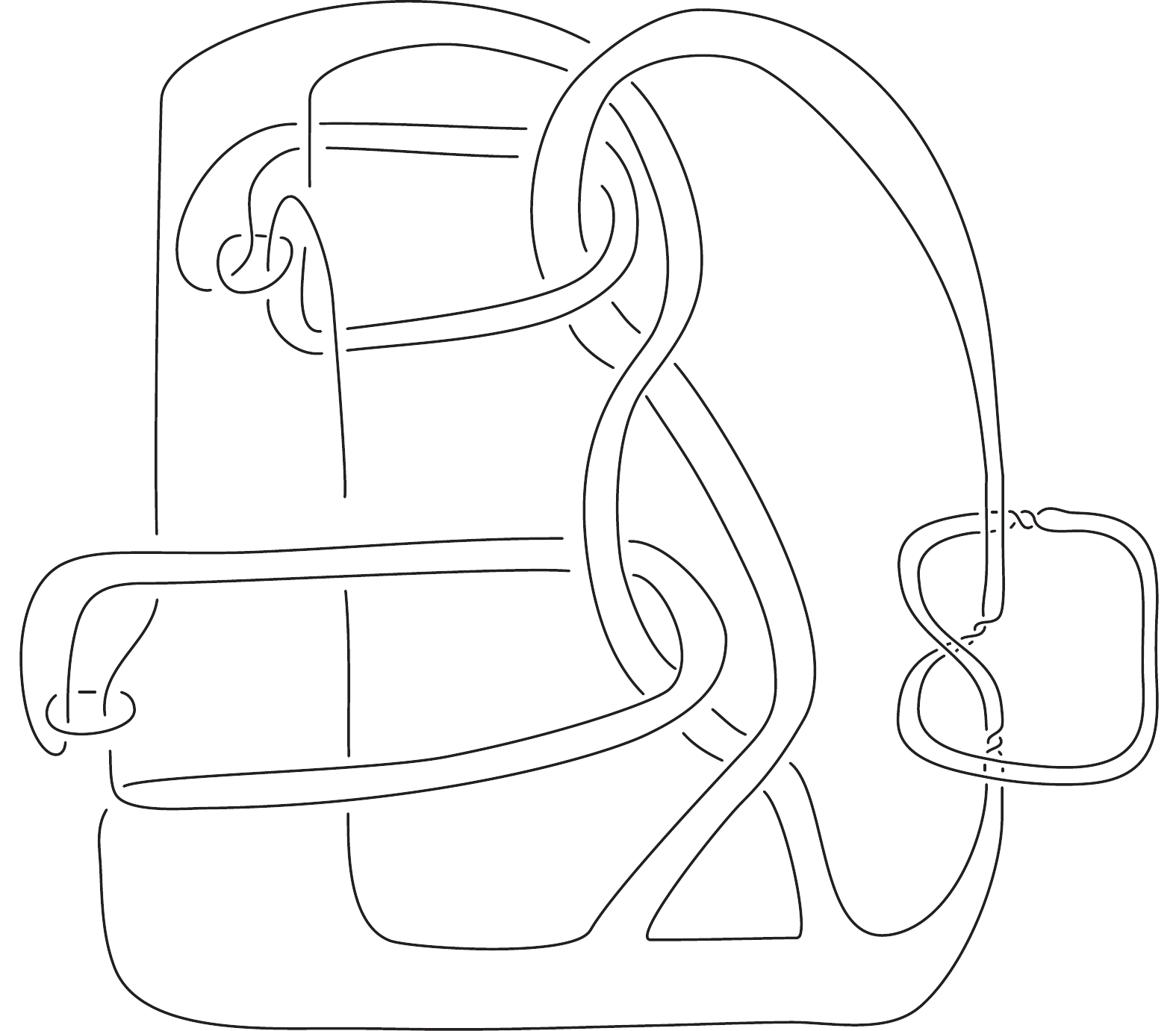}
			\end{picture}
			\caption{Tying one band of $9_{46}$ into a trefoil}
			\label{addtrefoil}
		\end{center}
	\end{figure}
	
	Second, we must find an infinite set $\mathcal{P}$ of knots $P_{m}$ such that no non-trivial rational linear combination of $\left\{\rho_{0}\left(P_{m}\right)\right\}$ is a rational multiple of $\rho^{1}(R)$.  These $\rho_{0}$ invariants are discussed in Lemma~\ref{indepsigs}.  Lemma~\ref{indepsigs} provides an infinite subset of the knots $P_{m}$ defined in Section~\ref{infectingknots} that satisfy this property.
	
	For each $n\geq 0$, we define a family of knots $\left\{K^{n}_{m}: m\geq 1\right\}\subset\G_{n+2}$.  This family is defined as in Section~\ref{ourexamples}, except using the knot $R$ as the seed knot.  It should be noted that (possibly) having tied a trefoil in one of the bands does not affect the conclusions of any of the results in Section~\ref{highgropes}.  In particular, these $K^{n}_{m}$ lie in $\G_{n+2}$.
	
	Now we must prove that no non-trivial linear combination of the knots $K^{n}_{m}$ is (even rationally) $(n.5)$-solvable.  Since $\G_{n+2.5}\subset\F_{n.5}$, it will follow that no such combination of these knots lies in $\G_{n+2.5}$, hence these knots are linearly independent in $\G_{n+2}/\G_{n+2.5}$.  Suppose that a finite sum $\displaystyle\widetilde{K}=\#_{m=1}^{\infty} j_{m}\,K^{n}_{m}$ is rationally $(n.5)$-solvable.  We may assume that $j_{1}\geq 0$.  We construct a faimly of $4$-manifolds $W_{i}$ and reach a contradiction.  Let us abbreviate $M_{K^{n}_{m}}=M^{n}_{m}$.
	
	Consider $W_{n}$ with boundary $M^{0}_{1}\sqcup M^{0}_{1}\sqcup M_{R}$.  Recall that $M^{0}_{1}=M_{K^{0}_{1}}=M_{P_{1}}$.  Let $\pi=\p\left(W_{n}\right)$ and consider $\phi:\pi\to \pi/\pi^{(n+1)}_{r}$.  Then by property (3) of Proposition~\ref{chlprop}, 
	
	\begin{equation}\label{theequation}
		\rho\left(M_{P_{1}},\phi_{\a'}\right)+\rho\left(M_{P_{1}},\phi_{\b'}\right)+\rho\left(M_{R},\phi_{R}\right)=\rho\left(\del W_{n},\phi\right)=-\sum_{m}C_{m}\,\rho_{0}\left(P_{m}\right)
	\end{equation}
	where $C_{1}\geq 0$ and $\phi_{R}:\p(M_{R})\to\pi/\pi^{(n+1)}_{r}$ is induced by inclusion (and similarly for $\phi_{\a'}$ and $\phi_{\b'}$).  By property (2) of Proposition\ref{chlprop}, $j_{\ast}\left(\p\left(M_{P_{1}}\right)\right)\subset \pi^{(n)}$, implying that the restrictions $\phi_{\a'}$ and $\phi_{\b'}$ factor through the respective abelianizations.  Additionally by property (2), at least one of these coefficient systems is non-trivial.  Hence, by Proposition~\ref{chlprop}, $$\rho\left(M_{P_{1}},\phi_{\a'}\right)+\rho\left(M_{P_{1}},\phi_{\b'}\right)=\epsilon\,\rho_{0}\left(P_{1}\right)$$ where $\epsilon$ is equal to one or two.  Thus we can simplify equation~\ref{theequation} to yield
	
	\begin{equation}\label{theequation2}
		\left(\epsilon+C_{1}\right)\rho_{0}\left(P_{1}\right)+\sum_{m>1}C_{m}\,\rho_{0}\left(P_{m}\right)=-\rho\left(M_{R},\phi_{R}\right)
	\end{equation}
	where $\epsilon+C_{1}\geq 1$.  Also by property (2) of Proposition~\ref{chlprop}, $$j_{\ast}\left(\p\left(M_{R}\right)\right)\subset \pi^{(n-1)}$$
	so $\phi_{R}$ factors through $G/G^{(2)}$ where $G=\pi\left(M_{R}\right)$.  Since $\ker(\phi_{R})\subset G^{(1)}$, $\phi_{R}$ is determined by the kernel, denoted $A$, of the induced map $$\overline{\phi_{R}}: G^{(1)}/G^{(2)}\to \textrm{im}(\phi_{R})$$  $A$ is a submodule of the Alexander module $G^{(1)}/G^{(2)}$.  Since the Alexander polynomial of $R$ is $(2t-1)(t-2)$, the product of two irreducible coprime factors, the Alexander module of $R$ has only four submodules: the trivial submodule, $A$ itself, and $\langle\alpha'\rangle$ and $\langle\b'\rangle$.  One can examine these four possibilities for $A$ and conclude that $\rho\left(M_{R},\phi_{R}\right)\in\left\{0,\rho^{1}(R)\right\}$.  Combining this with equation~\ref{theequation2}, we have $$\left(\epsilon+C_{1}\right)\rho_{0}\left(P_{1}\right)+\sum_{m>1}C_{m}\,\rho_{0}\left(P_{m}\right)=C_{0}\,\rho^{1}(R)$$ where $C_{0}\in\{0,1\}$.  Since $\epsilon+C_{1}\geq 1$, we have written a multiple of $\rho^{1}(R)$ as a non-trivial linear combination of $\left\{\rho_{0}\left(P_{m}\right)\right\}$, a contradiction.
\end{proof}

\begin{prop}\label{chlprop}(Proposition 8.2 of~\cite{CHL07b}).
		If $\widetilde K$ is rationally $(n.5)$-solvable, then for each $0\leq i\leq n$, there is a $4$-manifold $W_{i}$ with the following properties.  Let $\pi=\p(W_{i})$.
		\begin{enumerate}
			\item $W_{i}$ is a rational $(n)$-bordism where, for $i<n$, $\del W_{i}=M^{n}_{1}$ and $\del W_{n}=M^{0}_{1}\sqcup M^{0}_{1}\sqcup M_{R}$,
			\item under the inclusion(s) $j:M^{n-i}_{1}\subset\del W_{i}\to W_{i}$, $$j_{\ast}:\p\left(M^{n-i}_{1}\right)\subset \pi^{(i)},$$
			and for each $i$ (at least one of the copies of) $M^{n-i}_{1}\subset\del W_{i}$, $$j_{\ast}:\left(\p\left(M^{n-i}_{1}\right)\right)\cong\Z\subset \pi^{(i)}/\pi^{(i+1)}_{r}$$
			and under the inclusion $j:M_{R}\subset \del W_{n}\to W_{n}$, $$j_{\ast}:\left(\p\left(M_{R}\right)\right)\cong\Z\subset\pi^{(n-1)}/\pi^{(n)}_{r}$$
			\item for any PTFA coefficient system $\phi:\p\left(W_{i}\right)\to\Gamma$ with $\Gamma^{(n+1)}_{r}=1$, $$\rho\left(\del W_{i},\phi\right)=\sigma^{(2)}_{\Gamma}\left(W_{i,\phi}\right)-\sigma\left(W_{i}\right)=-\sum_{m}C_{m}\,\rho_{0}\left(P_{m}\right)$$
			for some integers $C_{m}$ (depending on $\phi$) with $C_{1}\geq 0$.
		\end{enumerate}
\end{prop}
	
	We refer the reader to~\cite{CHL07b} for a proof of this proposition.

\section{String links}

	We follow the notation of~\cite{sH08}.  Let $\C(m)$ denote the group of $m$-component string links modulo concordance (the group operation is concatenation).  If $L\in\C(m)$, let $\widehat L$ denote the closure of $L$; $\widehat L$ is obtained by attaching the $m$-component trivial link to the ends of $L$ in the obvious way.  Let $\G^{m}_{n}$ denote the subset of $\C(m)$ defined by $L\in\G^{m}_{n}$ if the components of $\widehat L$ bound disjoint Gropes of height $n$ in $D^{4}$.  It can be shown that the Grope filtration of $\C(m)$
	$$0\subset\cdots\subset\G^{m}_{n.5}\subset \G^{m}_{n}\subset\cdots\subset\G^{m}_{1.5}\subset\G^{m}_{1}\subset\C(m)$$
	is a filtration of $\C(m)$ by normal subgroups.  If $\B(m)$ denotes the sugroup of $\C(m)$ consisting of boundary string links, we define $\BG^{m}_{n}=\G^{m}_{n}\cap\B(m)$.
	
	There is also the $(n)$-solvable filtration of $\C(m)$, whose definition we omit, that is related to the Grope filtration by $\G^{m}_{n+2}\subset\F^{m}_{n}$ for all $n\in\frac{1}{2}\N$.  Let $\BF^{m}_{n}=\F^{m}_{n}\cap \B(m)$.
	
	\begin{rem}\label{salfeat}
		For each $n\in\Z$, Harvey~\cite[Definition 3.11]{sH08} defined a link invariant $\rho_{n}(L)\in\R$.  The salient features of this $\rho_{n}$ are listed below:
		\begin{enumerate}
			\item~\cite[Corollary 6.7]{sH08}  For each $n\geq 0$ and $m\geq 1$, $\rho_{n}:\B(m)\to\R$ is a homomorphism,
			\item~\cite[Theorem 4.4]{CH08} $\rho_{n}(L)$ vanishes for all $L\in\BF^{m}_{n.5}$.
			\item~\cite[Theorem 5.8]{sH08} If $T$ is the identity element of $\B(m)$, $\eta$ an unknotted circle in the complement of $T$ that lies in $\pi_{1}\left(D^{2}\times I - T\right)^{(n)}-\pi_{1}\left(D^{2}\times I - T\right)^{(n+1)}$, and $T(\eta,K)$ denotes $T$ infected along $\eta$ by the knot $K$, then $\rho_{n}(T(\eta,K))=\rho_{0}(K)$, where $\rho_{0}(K)$ is the classical signature discussed in Lemma~\ref{indepsigs}.
		\end{enumerate}
	\end{rem}
	
	Harvey proved that the Grope filtration of the string link concordance group is a non-trivial filtration.
	
	\begin{thm}[6.13 of~\cite{sH08} and 4.4 of~\cite{CH08}]
		For each $n\geq 1$ and $m\geq 2$, the abelianization of $\BG^{m}_{n}/\BG^{m}_{n+1.5}$ has infinite rank; hence $\BG^{m}_{n}/\BG^{m}_{n+1.5}$ is an infinitely generated subgroup of $\G^{m}_{n}/\G^{m}_{n+1.5}$.
	\end{thm}
	
	With our knots $P_{m}\in\G_{2}$, we are able to close the index gap.
	
	\begin{thm}\label{secondaryresult}
		For each $n\geq 2$ and $m\geq 2$, the abelianization of $\BG^{m}_{n}/\BG^{m}_{n+.5}$ has infinite rank; hence $\BG^{m}_{n}/\BG^{m}_{n.5}$ is an infinitely generated subgroup of $\G^{m}_{n}/\G^{m}_{n+.5}$.
	\end{thm}
	\begin{proof}
		For $n=2$, let $L_{k}$ denote the $m$-component boundary string link whose closure is the union of the $(m-1)$-component trivial link and the knot $P_{k}$ as defined in Section~\ref{infectingknots}.  Since $P_{k}\in\G_{2}$, the components of the closure of $L_{k}$ bound disjoint Gropes of height $2$, $m-1$ of which are discs.  Thus $L_{k}\in\BG^{m}_{2}$.  The linear independence of the $L_{k}$ in $\BG^{m}_{2}/\BG^{m}_{2.5}$ follows from the argument below when $n\geq 3$.
		
		Assume $n\geq 1$.  We prove the statement for $n+2$.   Let $F$ denote the free group on $m$ generators, and $T$ be the trivial $m$-component string link, so that $F\cong\pi_{1}\left(D^{2}\times I - T\right)$.  Pick an element $[\eta]\in F^{(n)}-F^{(n+1)}$.  By~\cite[Lemma 3.9]{CT07}, we may pick an embedded curve $\eta$ in $D^{2}\times I - T$ that bounds a disk in $D^{2}\times I$ and that bounds a Grope of height $n$ in $D^{2}\times I - T$.  Let $S_{\eta}=\{T(\eta,P_{k}):\ P_{k}\mbox{ as defined in Section~\ref{infectingknots}}\}$ where $T(\eta,K)$ is defined by item 3 of Remark~\ref{salfeat}.  Since $\eta\in F^{(1)}$, there are disjoint surfaces bounded by the components of $T$ that miss the curve $\eta$.  Thus each $L\in S_{\eta}$ is a boundary string link.  Since $P_{k}\in\G_{2}$, we have that if $L\in S_{\eta}$, then $\widehat L\in\G^{m}_{n+2}$.  Thus, $S_{\eta}\subset\BG^{m}_{n+2}$.
		
		By items 1 and 2 of Remark~\ref{salfeat}, $\rho_{n+2}:\BG^{m}_{n+2}/\BG^{m}_{n+2.5}\to\R$ is a well-defined homomorphism.  By item 3 of Remark~\ref{salfeat}, the image of $S_{\eta}$ under this homomorphism is $\rho_{n+2}\left(S_{\eta}\right)=\{\rho_{0}(P_{k})\}$, which contains a $\Q$-linearly independent subset of $\R$ by Lemma~\ref{indepsigs}.  The conclusion follows.
	\end{proof}
		
\section{Caclulation of $\rho$-invariants}

	We now calculate the classical signatures of the $P_{m}$ and discuss why they are linearly independent.

\begin{lem}\label{sigcalc}
	The $0$-th-order signature of $P_{m}$ is given by $\displaystyle\rho_{0}\left(P_{m}\right)=2-2\,\frac{\theta_{m}}{\pi}$ where $\theta_{m}$ is the number satisfying $\displaystyle\cos\left(\theta_{m}\right)=\frac{2\sqrt[3]{m}-1}{2\sqrt[3]{m}}$ and $0<\theta_{m}<\pi$.
\end{lem}
\begin{proof}
	Let $E_{m}$ denote the exterior of $P_{m}$.  Recall that $P_{m}$ is defined by surgery on the link in Figures~\ref{Pm} and~\ref{firstgropepicture}.  Depicted in Figure~\ref{intermediate} is that link, isotoped for convenience.
	
	\begin{figure}[ht!]
		\begin{center}
			\begin{picture}(281, 186)(0,0)
					\put(83, 62){$\cdots$}
				     \put(258, 75){\tiny $H_{0}$}
				     \put(258, 142){\tiny $H_{1}$}
				     \put(174, 35){\tiny $H_{2}$}
				     \put(220, 49){\tiny $H_{3}$}
				     \put(224, 85){\tiny $H_{4}$}
				     \put(194, 122){\tiny $H_{5}$}
				     \put(228, 122){\tiny $H_{6}$}
				     \put(192, 142){\tiny $H_{7}$}
				     \put(226, 142){\tiny $H_{8}$}
				     \put(130, 52){\tiny $H_{9}$}
				     \put(97, 122){\tiny $H_{10}$}
				     \put(159, 122){\tiny $H_{11}$}
				     \put(94, 142){\tiny $H_{12}$}
				     \put(157, 142){\tiny $H_{13}$}
				     \put(4, 122){\tiny $H_{6+4m}$}
				     \put(2, 142){\tiny $H_{8+4m}$}
				     \put(50, 171){\tiny $H_{9+4m}$}
				     \put(54, 169){\vector(2,-3){7}}
				     \put(80, 170){\tiny $H_{7+4m}$}
				     \put(88, 166){\vector(-1,-4){9}}
				\put(20, 45){$\underbrace{\hspace{4.7cm}}$}
				\put(84, 15){$m$}
				\put(240, 180){$U$}
				\includegraphics{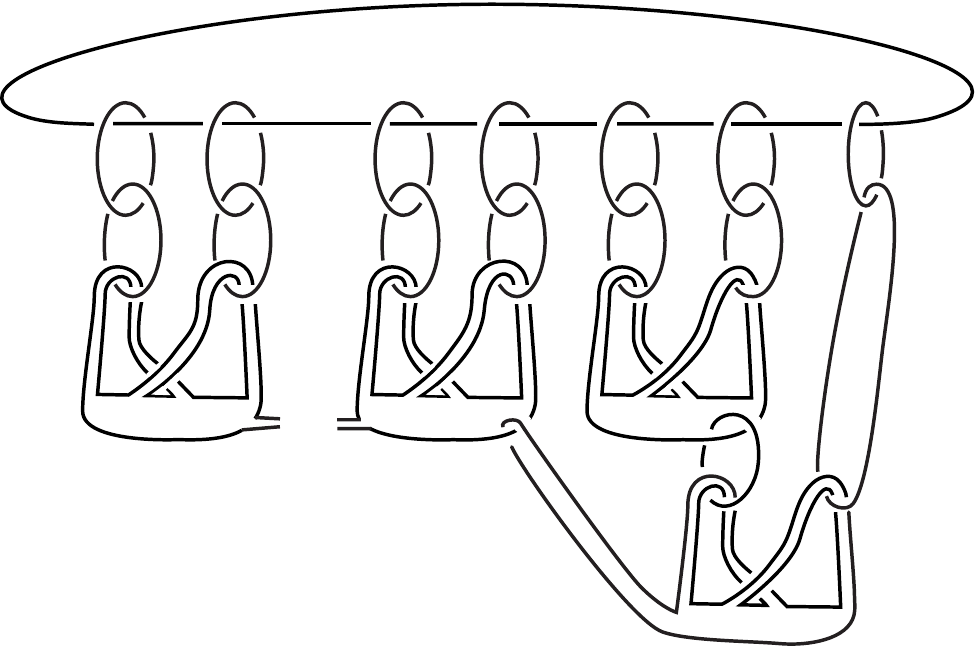}
			\end{picture}
			\caption{The surgery link defining $P_{m}$}
			\label{intermediate}
		\end{center}
	\end{figure}
	
	After performing the handle cancellations
	$$H_{1}-H_{0},\ H_{7}-H_{5},\ H_{8}-H_{6},\ H_{12}-H_{10},\ H_{13}-H_{11},\ H_{8+4m}-H_{6+4m},\ H_{9+4m}-H_{7+4m},$$ we arrive at the link in Figure~\ref{alexmodulecalc1}.  Thus, we can describe $E_{m}$ as the image of $S^{3}-U$ after $0$-surgery on the two component link $\{H_{9},H_{2}\}$ as in Figure~\ref{alexmodulecalc1}.  This link is the result after performing $0$-surgery on all components of the link in Figure~\ref{Pm} except for the bottom-most two.  As discussed in the proof of Proposition~\ref{grope}, the link $\{H_{9},H_{2}\}$ is a link in $S^{3}$.
	\begin{figure}[ht!]
		\begin{center}
			\begin{picture}(280, 150)(0,0)
				\put(85, 25){\small$\cdots$}
				\put(20, 10){$\underbrace{\hspace{4.7cm}}$}
				\put(83, -4){$m$}
				\put(280, 133){$U$}
				\put(262, 100){\small $H_{2}$}
				\put(12, 80){\small $H_{9}$}
				\includegraphics{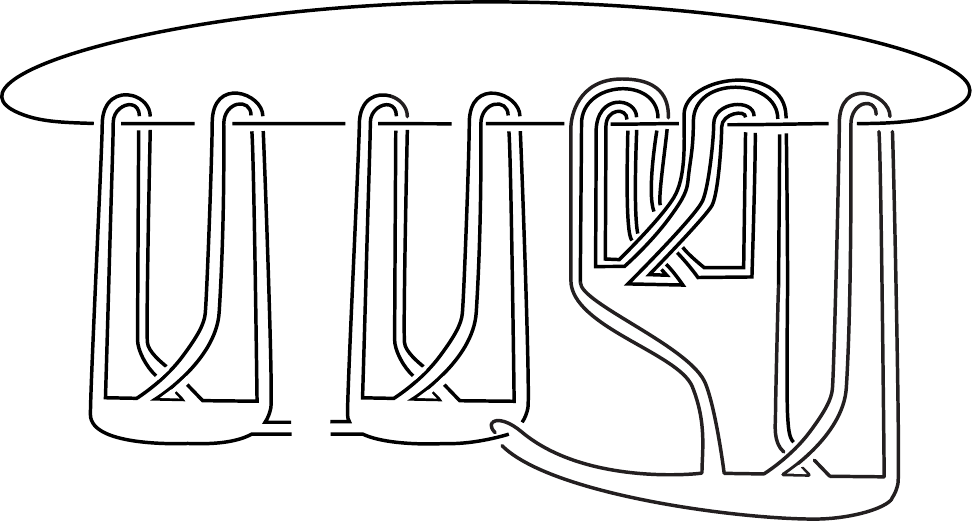}
			\end{picture}
			\caption{Zero surgery on the link $\{H_{2},H_{9}\}$ yields $E_{m}$}
			\label{alexmodulecalc1}
		\end{center}
	\end{figure}
	
	As in~\cite[\S 7.C]{dR76}, we can compute the Alexander module of $E_{m}$ using these surgery instructions.  Let $\a$ denote a lift of $H_{9}$ in the infinite cyclic cover of $U$, and let $\b$ denote the lift of $H_{2}$ that links $\a$ once.  Let $\tau$ denote a generator of the group of covering translations.  Recall that the matrix $$\lambda=\left(\begin{array}{cc}
	a+\sum_{i\in\Z-\{0\}} \lk(\a,\tau^{i}\a)\,t^{i}&\sum_{i\in\Z} \lk(\a,\tau^{i}\b)\,t^{i}\\
	\sum_{i\in\Z} \lk(\b,\tau^{i}\a)\,t^{i}&b+\sum_{i\in\Z-\{0\}} \lk(\b,\tau^{i}\b)\,t^{i}
	\end{array}\right)$$
where $\displaystyle a+\sum_{i\in\Z-\{0\}} \lk(\a,\tau^{i}\a)$ equals the framing of $H_{9}$ and $\displaystyle b+\sum_{i\in\Z-\{0\}} \lk(\b,\tau^{i}\b)$ equals the framing of $H_{2}$, is a presentation matrix for the Alexander module of $E_{m}$.
	
	Displayed in Figure~\ref{alexmodulecalc2} is the link $\{U,H_{9},H_{2}\}$ after an isotopy.  Figure~\ref{alexmodulecalc3} depicts the infinite cyclic cover of $E_{m}$.
	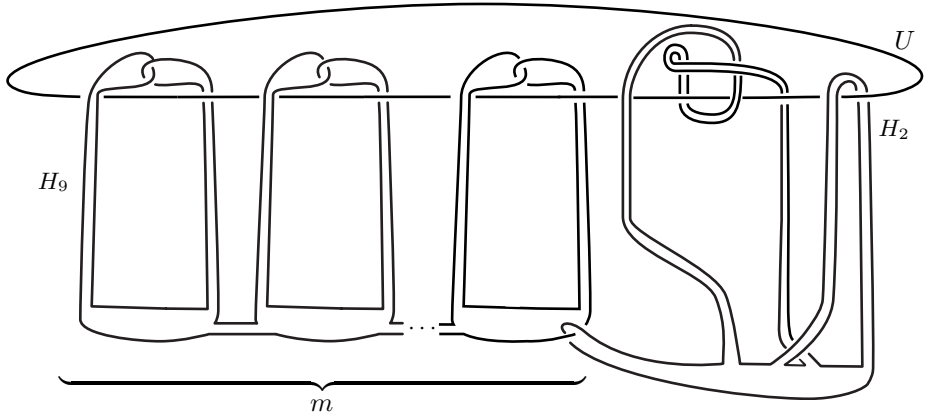
\begin{figure}[ht!]
		\begin{center}
			\begin{picture}(350, 150)(0,0)
				\put(152, 25){\small$\cdots$}
				\put(20, 10){$\underbrace{\hspace{7cm}}$}
				\put(115, -4){$m$}
				\put(336, 133){$U$}
				\put(330, 100){\small $H_{2}$}
				\put(12, 80){\small $H_{9}$}
				\includegraphics{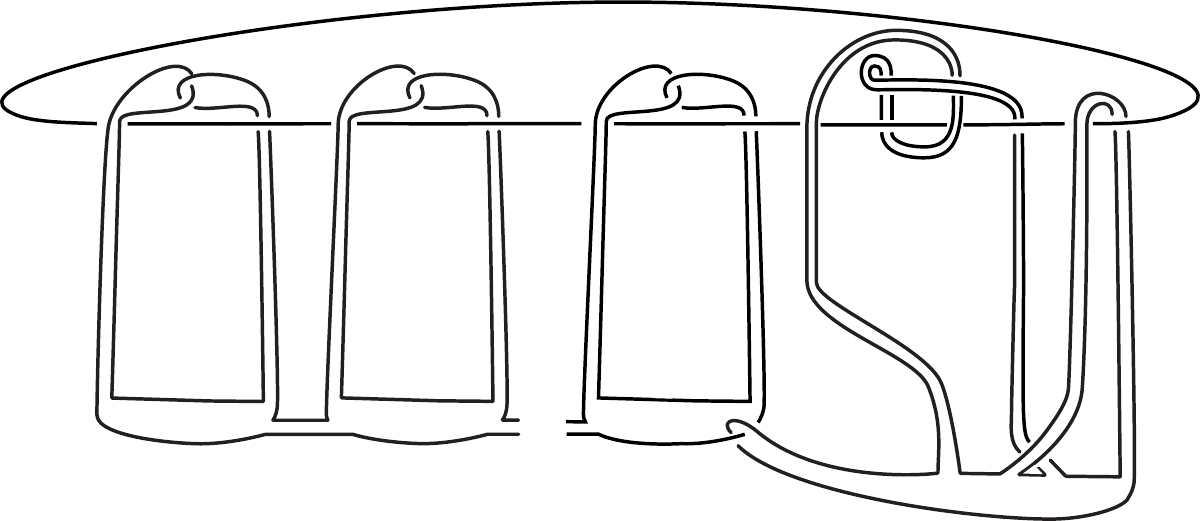}
			\end{picture}
			\caption{An isotopic image of the link from Figure~\ref{alexmodulecalc1}}
			\label{alexmodulecalc2}
		\end{center}
	\end{figure}
	
	\begin{figure}[ht!]
		\begin{center}
			\begin{picture}(254, 341)(0,0)
					\put(73, 29){\small$\cdots$}
					\put(73, 144){\small$\cdots$}
					\put(73, 259){\small$\cdots$}
					\put(260, 215){\small $\b$}
					\put(257, 195){$0$}
					\put(260, 100){\small $\tau\b$}
					\put(257, 80){$0$}
					\put(260, 15){\small $\tau^{2}\b$}
					\put(258, -0){$0$}
					\put(-18, 55){\small $\tau^{2}\a$}
					\put(-6, 35){$0$}
				 	\put(-13, 170){\small $\tau\a$}
					\put(-6, 150){$0$}
					\put(-8, 285){\small $\a$}
					\put(-6, 265){$0$}
					\put(280, 218){\vector(0,-1){115}}
					\put(275, 215.5){---}
					\put(290, 160){$\tau$}
					\includegraphics{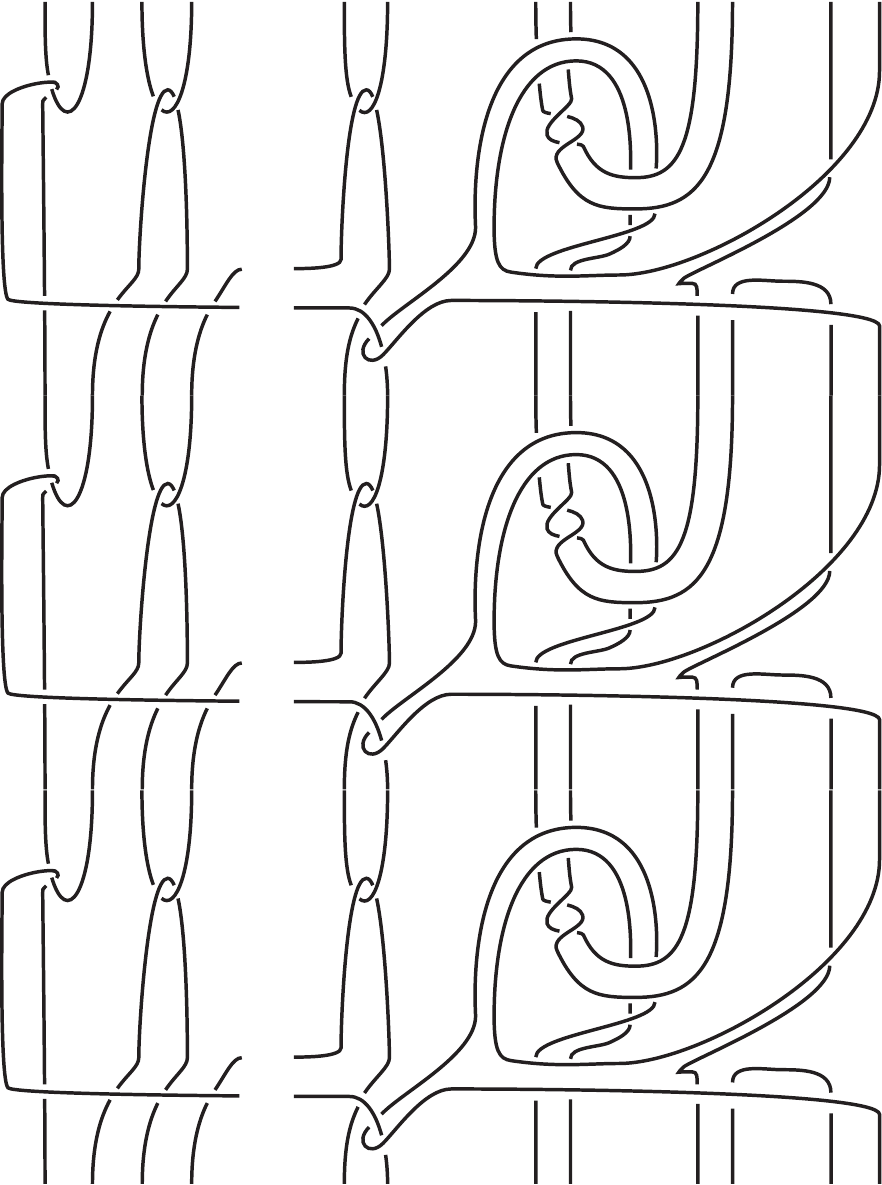}
			\end{picture}
			\caption{The infinite cyclic cover of $E_{m}$}
			\label{alexmodulecalc3}
		\end{center}
	\end{figure}
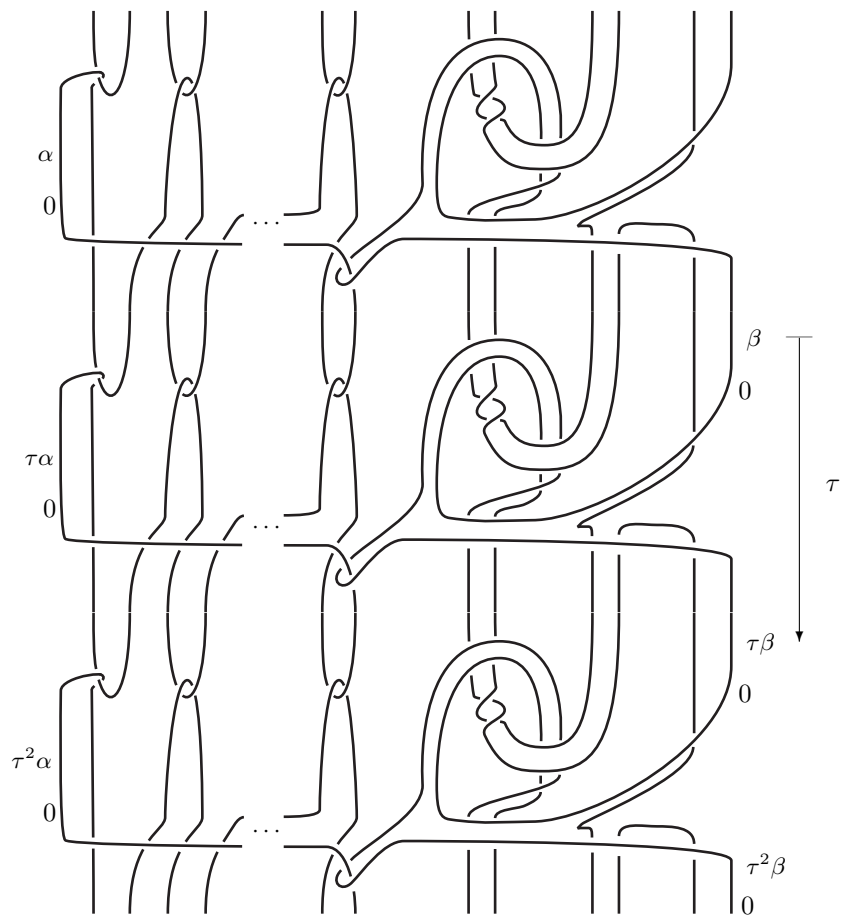
	
	From Figure~\ref{alexmodulecalc3}, one computes the Alexander matrix to be $$\lambda=\left(\begin{array}{cc}
	m\left(2-t-t^{-1}\right)&1\\
	1&t^{2}+t^{-2}-4\left(t+t^{-1}\right)+6
	\end{array}\right)$$
	
	The Alexander polynomial of $P_{m}$ is $$\Delta_{m}(t)=\det\lambda=m\left(\left(t^{3}+t^{-3}\right)-6\left(t^{2}+t^{-2}\right)+15\left(t+t^{-1}\right)-\frac{20m-1}{m}\right)$$
	
	Recall that $\rho_{0}$ of a knot is equal to the integral of the Levine-Tristram signature function over the unit circle (normalized to have length one), and that this function is constant away from the roots of the Alexander polynomial.  We claim that the polynomial $\Delta_{m}(t)$ has exactly two roots on the unit circle, namely $e^{\pm i\theta_{m}}$, where $\displaystyle\cos\left(\theta_{m}\right)=\frac{2\sqrt[3]{m}-1}{2\sqrt[3]{m}}$ and $0<\theta_{m}<\pi$.  After proving this claim, we can conclude the signature of $\lambda$ is $+2$ when substituting $t=-1$, thus $\displaystyle\rho_{0}\left(P_{m}\right)=2-2\,\frac{\theta_{m}}{\pi}$.
	
	It is easy to verify that $e^{\pm i\theta_{m}}$ are roots of $\Delta_{m}(t)$ and that $\Delta_{m}(t)$ factors as 

	\begin{equation}\label{alexfactor}
	\Delta_{m}(t)\ =\ m\left(t-\frac{2\sqrt[3]{m}-1}{\sqrt[3]{m}}+t^{-1}\right)\left(t^{2}+t^{-2}-\left(4+\frac{1}{\sqrt[3]{m}}\right)\left(t+t^{-1}\right)+6+\frac{2}{\sqrt[3]{m}}+\frac{1}{\sqrt[3]{m^{2}}}\right)
	\end{equation}

	Assume that $e^{i\psi}$, where $\psi\in\R$, is a root of the second factor in Equation~(\ref{alexfactor}).  Then $\psi$ must satisfy the equation $$4(\cos\psi)^{2}-\left(8+\frac{2}{\sqrt[3]{m}}\right)\cos\psi+4+\frac{1}{\sqrt[3]{m^{2}}}+\frac{2}{\sqrt[3]{m}}=0$$  The quadratic formula yields $\displaystyle\cos\psi=\frac{4\sqrt[3]{m}+1\pm i\sqrt{3}}{4\sqrt[3]{m}}$.  This implies that $\psi\not\in\R$, and we conclude that $\Delta_{m}(t)$ has no other roots on the unit circle.
\end{proof}

\begin{lem}\label{indepsigs}
	There is a subsequence of $P_{m_{j}}$ of the knots $P_{m}$ whose  set of $0$-th-order signatures is linearly independent over $\Q$.
\end{lem}
\begin{proof}
	By Lemma~\ref{sigcalc}, it is sufficient to find a set $\{\theta_{m_{j}}\}_{j=1}^{\infty}$ that is linearly independent over $\Q$.  Let $5=p_{1}< p_{2}<p_{3}<\cdots$ be a sequence of prime integers satisfying $p_{j}\equiv 1\pmod4$ for all $j$, and set $m_{j}=\frac{1}{64}\left(p_{j}-1\right)^{6}\in\Z$.  It was shown in~\cite[Proposition 2.6]{COT04} that $\{\theta_{m_{j}}\}_{j=1}^{\infty}$ is linearly independent over $\Q$.
\end{proof}

	\bibliographystyle{amsalpha}
	\bibliography{GropeFiltration}

\end{document}